\documentclass[a4paper, 11pt]{article}
\usepackage[utf8]{inputenc}
\usepackage[english]{babel}
\usepackage{amssymb,amsmath,amsthm,amsfonts}

\usepackage{indentfirst}
\usepackage{comment}
\usepackage{enumerate}
\usepackage{enumitem}

\usepackage[pdftex]{graphicx}
\graphicspath{{pictures/}}
\DeclareGraphicsExtensions{.pdf,.png,.jpg}

\usepackage{geometry}
\newtheorem{theorem}{Theorem}
\newtheorem{lemma}[theorem]{Lemma}

\newtheorem{remark}{Remark}

\DeclareMathOperator*{\argmin}{argmin}
\DeclareMathOperator*{\prox}{prox}

\usepackage{float} 
\usepackage{caption}
\usepackage{subcaption}

\usepackage[usenames]{color}
\usepackage{colortbl}

\title{A fast continuous time approach with time scaling for nonsmooth convex optimization}

\author{Radu Ioan Bo\c t
\footnote{Faculty of Mathematics, University of Vienna, Oskar-Morgenstern-Platz 1, 1090 Vienna, Austria, 
{email: \tt radu.bot@univie.ac.at}. Research partially supported by FWF (Austrian Science Fund), project P 34922-N.} \and
Mikhail A. Karapetyants 
\footnote{Faculty of Mathematics, University of Vienna, Oskar-Morgenstern-Platz 1, 1090 Vienna, Austria, 
{email: \tt mikhail.karapetyants@univie.ac.at.} Research supported by the Doctoral Programme \emph{Vienna Graduate School on Computational Optimization (VGSCO)} which is funded by FWF (Austrian Science Fund), project W 1260.} 
}

\begin{document}
\maketitle

\begin{abstract}
    
In a Hilbert setting we study the convergence properties of a second order in time dynamical system combining viscous and Hessian-driven damping with time scaling in relation with the minimization of a nonsmooth and convex function. The system is formulated in terms of the gradient of the Moreau envelope of the objective function with time-dependent parameter. We show fast convergence rates for the Moreau envelope and its gradient along the trajectory, and also for the velocity of the system. From here we derive fast convergence rates for the objective function along a path which is the image of the trajectory of the system through the proximal operator of the first. Moreover, we prove the weak convergence of the trajectory of the system to a global minimizer of the objective function. Finally, we provide multiple numerical examples which illustrate the theoretical results.



\end{abstract} 

\section{Introduction}

Let $H$ be a real Hilbert space endowed with the scalar product $\langle \cdot, \cdot \rangle$ and norm $\| x \| = \sqrt{\langle x, x \rangle}$ for $x \in H$. In connection with the minimization problem
\[
\min_{x \in H} \Phi(x),
\]
we will study the asymptotic behaviour of the second order in time evolution equation 
\begin{equation}\label{Syst}
\ddot x(t) + \frac{\alpha}{t} \dot x(t) + \beta(t) \frac{d}{dt} \nabla \Phi_{\lambda(t)}(x(t)) + b(t) \nabla \Phi_{\lambda(t)}(x(t)) = 0,
\end{equation}
with initial conditions $x(t_0) = x_0 \in H$, $\dot x(t_0) = u_0 \in H$, where  $\alpha \geq 1 $, $t_0 >0$, and $\beta: [t_0, +\infty) \longrightarrow [0, +\infty)$ and $ b,  \lambda: [t_0, +\infty) \longrightarrow (0, +\infty)  $ are differentiable functions. 

We assume that $\Phi: H \longrightarrow \overline{\mathbb{R}} = \mathbb{R} \cup \{\pm \infty\} $ is a proper, convex and lower semicontinuous function and denote by $\Phi_{\lambda}: H \longrightarrow \mathbb{R} $ its Moreau envelope of parameter $\lambda > 0$. In addition, we assume that $\argmin \Phi$,  the set of global minimizers of $\Phi$, is not empty and denote by $\Phi^*$ the optimal objective value of $\Phi$.

Our aim is to derive rates of convergence for the Moreau envelope of the objective function and the objective function itself to $\Phi^*$, as well as for the gradient of the Moreau envelope of the objective function and the velocity of the trajectory to zero in terms of the Moreau parameter function $\lambda$ and the time scaling function $b$. In addition, we will provide a setting which also guarantees the weak convergence of the trajectory of the dynamical system to a minimizer of $\Phi$. The theoretical results will be illustrated by multiple numerical experiments.

\subsection{Historical remarks}

Inertial dynamics were introduced by Polyak in \cite{P} in form of the so-called heavy ball with friction method
\[
\ddot x(t) + \alpha \dot x(t) + \nabla \Phi(x(t)) = 0,
\]
with fixed viscous coefficient $\alpha >0$, in order to accelerate the gradient method for the minimization of a continuous differentiable function $\Phi: H \to \mathbb{R}$. This system was later studied by Alvarez-Attouch \cite{AA_0, AA} and by Attouch-Goudou-Redont \cite{AGR}. For a convex function $\Phi$ an asymptotic convergence rate of $\Phi(x(t))$ to $\Phi^*$ of order $O\left( \frac{1}{t} \right)$ as $t \to +\infty$, as well as an improvement for a strongly convex function $\Phi$ to an exponential rate of convergence were proved. The weak convergence of the trajectories to a minimizer of $\Phi$ was also established.

A major step to obtain faster asymptotic convergence in the convex regime was done by Su-Boyd-Candes \cite{SBC}, by considering in the second order dynamical system an asymptotic vanishing damping coefficient
\begin{equation}\label{CN}
\ddot x(t) + \frac{\alpha}{t} \dot x(t) + \nabla \Phi(x(t)) = 0,
\end{equation}
for $t \geq t_0$ and $\alpha \geq 3$. Second order dynamical systems with variable and vanishing damping coefficients for optimization were studied, for instance, in \cite{BC, CEG, CEG_1}. The system \eqref{CN} corresponds to a continuous version of Nesterov's accelerated gradient method  \cite{N}. For the function values, rates of convergence of
\[
\Phi(x(t)) - \Phi^* = O\left( \frac{1}{t^2} \right) \ \text{as} \ t \to +\infty
\]
were obtained. For $\alpha >3$, in \cite{ACPR}  it was shown that the trajectory of \eqref{CN} converges weakly to an element of $\argmin \Phi$, and in \cite{AP, M} the asymptotic convergence rate of the function values was improved to $o\left( \frac{1}{t^2} \right)$ as $t \to +\infty$.

The following system which combines asymptotic vanishing damping with Hessian-driven damping was proposed by Attouch-Peypouquet-Redont in \cite{APR}
\begin{equation}\label{00}
\ddot x(t) + \frac{\alpha}{t} \dot x(t) + \beta \nabla^2 \Phi(x(t)) \dot x(t) + \nabla \Phi(x(t)) = 0
\end{equation}
for $t \geq t_0 $, where $\Phi: H \longrightarrow \mathbb{R}$ twice continuously differentiable and convex, $\alpha \geq 3$ and $\beta \geq 0$. Hessian-driven damping has a natural link with Newton’s method and gives rise to dynamical inertial Newton systems \cite{AABR}. The system \eqref{00} preserves the convergence properties of \eqref{CN}, while having for $\beta >0$
other important features, namely, 
$$\lim_{t \rightarrow +\infty} \|\nabla \Phi(x(t))\| =0 \ \mbox{and} \ \int_{t_0}^{+\infty} t^2 \|\nabla \Phi(x(t))\|^2dt < +\infty.$$
In addition, possible oscillations exhibited by the solutions of  \eqref{CN} are neutralized by \eqref{00}. 

\subsection{Time scaling}

Time scaling of the dynamical system \eqref{CN} was used in order to accelerate the rate of convergence of the values of the function $\Phi$ along the trajectory. The system \eqref{CN} becomes through time scaling a dynamical system of the form
\begin{equation}\label{timescaling}
\ddot x(t) + \frac{\alpha}{t} \dot x(t) + b(t) \nabla \Phi(x(t)) = 0,
\end{equation}
where $\alpha \geq 3$ and $b : [t_0, +\infty) \longrightarrow (0, +\infty)$ is a continuous scalar function, as it was introduced and studied by Attouch-Chbani-Riahi in \cite{ACR}. For \eqref{timescaling} it was shown that
\[
\Phi(x(t)) - \Phi^* = O\left( \frac{1}{t^2 b(t)} \right) \text{ as } t \to +\infty, 
\]
a convergence rate which can be improved to $o\left( \frac{1}{t^2b(t)} \right)$ as $t \to +\infty$, if $\alpha >3$.

In \cite{ACFR} (see also \cite{ABCR}) the dynamical system
\begin{equation}\label{h}
\ddot x(t) + \frac{\alpha}{t} \dot x(t) + \beta(t) \nabla^2 \Phi(x(t)) \dot x(t) + b(t) \nabla \Phi(x(t)) = 0,
\end{equation}
which combines viscous and Hessian-driven damping with time scaling, where $\alpha \geq 1$ and $\beta, b : [t_0, +\infty) \longrightarrow (0, +\infty)$ 
are functions with appropriate differentiability properties, was investigated. A quite general setting formulated in terms of the dynamical system parameter functions was identified in which the properties of \eqref{h}  concerning the convergence of the function values are preserved, while the gradient of $\Phi$ strongly converges along the trajectory to zero and the trajectory converges weakly to a minimizer of the objective function. In \cite{ACFR} a numerical algorithm obtained via time discretization of \eqref{h} was also introduced, exhibiting analogous convergence properties to the dynamical system.

\subsection{Nonsmooth optimization}

The Moreau envelope of a proper, convex and lower semicontinuous function $\Phi : H \to \overline{\mathbb{R}}$ has played a significant role in the literature when designing continuous-time approaches and numerical algorithms for the minimization of $\Phi$. This is defined as 
\begin{equation*}
\Phi_\lambda: H \to \mathbb{R}, \quad \Phi_{\lambda} (x) \ = \ \inf_{y \in H} \left\{ \Phi(y) + \frac{1}{2 \lambda} \| x - y \|^2 \right\},
\end{equation*}
where $\lambda > 0$ is called the parameter of the Moreau envelope (see, for instance, \cite{BC_book}). For every $\lambda >0$,  the functions $\Phi$ and $\Phi_{\lambda}$ share the same optimal objective value and the same set of minimizers. In addition, $\Phi_\lambda$ is convex and continuously differentiable with
\begin{align}\label{Morprox}
\nabla \Phi_{\lambda} (x) = \frac{1}{\lambda} ( x - \prox\nolimits_{\lambda \Phi} (x)) \quad \forall x \in H,
\end{align}
and $\nabla \Phi_\lambda$ is $\frac{1}{\lambda}$-Lipschitz continuous. Here,
$$\prox\nolimits_{\lambda \Phi}: H \to H, \quad \prox\nolimits_{\lambda \Phi} (x) = \argmin_{y \in H} \left\{ \Phi(y) + \frac{1}{2 \lambda} \| x - y \|^2 \right\},$$
denotes the proximal operator of $\Phi$ of parameter $\lambda$. For every $x \in H$ and $\lambda, \mu >0$ we have
\begin{equation}\label{proxineq}
\|\prox\nolimits_{\lambda \Phi}(x) - \prox\nolimits_{\mu \Phi}(x) \| \leq |\lambda  - \mu| \|\nabla \Phi_{\lambda} (x)\|.
\end{equation}

On the other hand, for every $x \in H$, the function $\lambda \in (0, +\infty) \to \Phi_\lambda (x)$ is nonincreasing  and differentiable, namely, 
\[
\frac{d}{d \lambda} \Phi_{\lambda} (x) = -\frac{1}{2} \| \nabla \Phi_{\lambda} (x) \|^2 \quad \forall \lambda >0.
\]
Attouch-Cabot considered in \cite{AC} (see also \cite{AP0} for a more general approach for monotone inclusions) in connection with the minimization of the proper, convex and lower semicontinuous function $\Phi : H \to \overline{\mathbb{R}}$ the following second order differential equation 
\begin{equation}\label{NSyst}
\ddot x(t) + \frac{\alpha}{t} \dot x(t) + \nabla \Phi_{\lambda(t)} (x(t)) = 0
\end{equation}
for $t \geq t_0$, where $\alpha \geq 1$ and $\lambda: [t_0, +\infty) \longrightarrow (0, +\infty)$ is continuously differentiable and non-decreasing. Convergence rates for the values of the Moreau envelope as well as for the velocity of the system were obtained
\[
\Phi_{\lambda(t)}(x(t)) - \Phi^* = o\left( \frac{1}{t^2} \right) \text{ and } \| \dot x(t) \| = o\left( \frac{1}{t} \right) \text{ as } t \to +\infty,
\]
from where convergence rates for the  $\Phi$ along $x(t)$ were deduced
\[
\Phi \big( \prox\nolimits_{\lambda(t) \Phi} (x(t)) \big) - \Phi^* = o\left( \frac{1}{t^2} \right) \text{ and } \| \prox\nolimits_{\lambda(t) \Phi} (x(t)) - x(t) \| = o\left( \frac{\sqrt{\lambda(t)}}{t} \right) \text{ as } t \to +\infty.
\]
In addition, the weak convergence of the trajectories $x(t)$ to a minimizer of $\Phi$ as $t \to +\infty$ was established.

Attouch-L\'aszl\'o considered in \cite{AL} in the same context the dynamical system
\begin{equation}\label{NSystHess}
\ddot x(t) + \frac{\alpha}{t} \dot x(t) + \beta \frac{d}{dt} \nabla \Phi_{\lambda(t)}(x(t)) + \nabla \Phi_{\lambda(t)}(x(t)) = 0
\end{equation}
where $\alpha > 1$ and $\beta > 0$, and the term $\frac{d}{dt} \nabla \Phi_{\lambda(t)}(x(t))$ is inspired by the Hessian driven damping. It was shown that for $\lambda(t) = \lambda t^2$, where $\lambda >0$, the system \eqref{NSystHess} inherits all major convergence properties of \eqref{NSyst} and, in addition, the following convergence rates for the gradient of the Moreau envelope of parameter $\lambda(t)$ and its time derivative along $x(t)$ were established
\[
\| \nabla \Phi_{\lambda(t)} (x(t)) \| = o\left( \frac{1}{t^2} \right) \text{ and } \left\| \frac{d}{dt} \nabla \Phi_{\lambda(t)} (x(t)) \right\| = o\left( \frac{1}{t^2} \right) \text{ as } t \to +\infty.
\]

\subsection{Our contribution}

In this paper, we derive a setting formulated in terms of $\alpha \geq 1$ and the parameter functions $\beta$, $b$ and 
$\lambda$ of the dynamical system \eqref{Syst} associated with the minimization of the proper, convex and lower semicontinuous function $\Phi : H \to \overline{\mathbb{R}}$, which allow us to prove

\begin{itemize}
    
\item  convergence rates for the Moreau envelope, its gradient and the velocity of the trajectory
\[
\Phi_{\lambda(t)}(x(t)) - \Phi^* = o \left( \frac{1}{t^2 b(t)} \right), \ \| \nabla \Phi_{\lambda(t)} (x(t)) \| \ = \ o \left( \frac{1}{t \sqrt{b(t) \lambda(t)}} \right) \text{ and } \| \dot x(t) \| = o \left( \frac{1}{t} \right)
\]
as $t \to +\infty$, respectively;

\item convergence rates for the objective function
\[
\Phi \big( \prox\nolimits_{\lambda(t) \Phi} (x(t)) \big) - \Phi^* = o \left( \frac{1}{t^2 b(t)} \right) \text{ and } \| \prox\nolimits_{\lambda(t) \Phi} (x(t)) - x(t) \| \ = \ o \left( \frac{\sqrt{\lambda(t)}}{t \sqrt{b(t)}} \right)
\]
as $t \to +\infty$;

\item the weak convergence of the trajectory $x(t)$ to a minimizer of $\Phi$ as $t \to +\infty$. 
\end{itemize}

In addition, we provide a particular formulation of the derived general setting for the case when the parameter functions are chosen to be polynomials and illustrate the influence of the latter on the convergence behaviour of the dynamical system by multiple numerical experiments.

\subsection{Existence and uniqueness of strong global solution}

This section is devoted to the topic of existence and uniqueness of a strong global solution of the system of our interest. To this aim we will rewrite \eqref{Syst} as a system of the first order in time equations in the product space $H \times H$. 

We assume first that $\beta : [t_0, +\infty) \longrightarrow [0, +\infty)$ is twice continuously differentiable with $\beta(t) >0$ for every $t \geq t_0$. We integrate \eqref{Syst} from $t_0$ to $t$ to obtain
\begin{align*}
\dot x(t) + \beta(t) \nabla \Phi_{\lambda(t)} (x(t)) + \int_{t_0}^t \left( \frac{\alpha}{s} \dot x(s) + b(s) \nabla \Phi_{\lambda(s)} (x(s)) \right) ds - \ \int_{t_0}^t \nabla \Phi_{\lambda(s)} (x(s)) \dot \beta(s) ds  & \\
 - \left( \dot x(t_0) + \beta(t_0) \nabla \Phi_{\lambda(t_0)} (x(t_0)) \right) & = 0.
\end{align*}

We denote $z(t) := \int_{t_0}^t \left( \frac{\alpha}{s}  \dot x(s) + \left( b(s) - \dot \beta(s) \right) \nabla \Phi_{\lambda(s)} (x(s)) \right) ds - \big( u_0 + \beta(t_0) \nabla \Phi_{\lambda(t_0)} (x_0)) \big)$ for every $t \geq t_0$. Since $\dot z(t) = \frac{\alpha}{t}\dot x(t) + \left( b(t) - \dot \beta(t) \right) \nabla \Phi_{\lambda(t)} (x(t))$ we notice, that \eqref{Syst} is equivalent  to
\begin{equation*}
\begin{cases}
&\dot x(t) + \beta(t) \nabla \Phi_{\lambda(t)} (x(t)) + z(t) = 0, \\
&\dot z(t) - \frac{\alpha}{t} \dot x(t) - \left( b(t) - \dot \beta(t) \right) \nabla \Phi_{\lambda(t)}(x(t)) = 0, \\
&x(t_0) = x_0, \ z(t_0) = -\left(u_0 + \beta(t_0) \nabla \Phi_{\lambda(t_0)} (x_0) \right).
\end{cases}
\end{equation*}
After multiplying the first line by $b(t) - \dot \beta(t)$ and the second one by $\beta(t)$, by summing them we get rid of the gradient of the Moreau envelope in the second equation
\begin{equation*}
\begin{cases}
&\dot x(t) + \beta(t) \nabla \Phi_{\lambda(t)} (x(t)) + z(t) = 0, \\
&\beta(t) \dot z(t) + \left( b(t) - \dot \beta(t) - \frac{\alpha \beta(t)}{t} \right) \dot x(t) + \left( b(t) - \dot \beta(t) \right) z(t) = 0,\\
&x(t_0) = x_0, \ z(t_0) = -\left(u_0 + \beta(t_0) \nabla \Phi_{\lambda(t_0)} (x_0) \right).
\end{cases}
\end{equation*}
We denote $y(t) = \beta(t) z(t) + \left( b(t) - \dot \beta(t) - \frac{\alpha \beta(t)}{t} \right) x(t)$, and, after simplification, we obtain  for the dynamical system the following equivalent formulation
\begin{equation*}
\begin{cases}
&\dot x(t) + \beta(t) \nabla \Phi_{\lambda(t)} (x(t)) + \left( \frac{\dot \beta(t) - b(t)}{\beta(t)} + \frac{\alpha}{t} \right) x(t) + \frac{1}{\beta(t)} y(t) = 0, \\
&\dot y(t) + \left( \ddot \beta(t) + \frac{3 b(t) \dot \beta(t) - 2 \dot \beta^2(t) - b^2(t)}{\beta(t)} + \frac{\alpha}{t} \left( b(t) - \dot \beta(t) - \frac{\beta(t)}{t} \right) - \dot b(t) \right) x(t) + \frac{b(t) - 2 \dot \beta(t)}{\beta(t)} y(t) = 0, \\
&x(t_0) = x_0, \ y(t_0) = -\beta(t_0) \left(u_0 + \beta(t_0) \nabla \Phi_{\lambda(t_0)} (x_0) \right) + \left( b(t_0) - \dot \beta(t_0) - \frac{\alpha \beta(t_0)}{t_0} \right) x_0.
\end{cases}
\end{equation*}

In case $\beta(t) = 0$ for every $t \geq t_0$, \eqref{Syst}  can be equivalently written as
\begin{equation*}
\begin{cases}
&\dot x(t)  - y(t) = 0, \\
&\dot y(t) + \frac{\alpha}{t} y(t) + b(t) \nabla \Phi_{\lambda(t)}(x(t)) = 0, \\
&x(t_0) = x_0, \ y(t_0) = u_0.
\end{cases}
\end{equation*}

Based on the two reformulation of the dynamical system \eqref{Syst} we can formulate the following existence and uniqueness result, which is a consequence of the Cauchy-Lipschitz theorem for strong global solutions. The result can be proved in the lines of the proofs of Theorem 1 in \cite{AL} or of Theorem 1.1 in \cite{APR} with some small adjustments.

\begin{theorem}
Suppose that $\beta : [t_0, +\infty) \longrightarrow [0, +\infty)$ is twice continuously differentiable  such that  either $\beta(t) >0$ for every $t \geq t_0$ or $\beta(t) = 0$ for every $t \geq t_0$,  and that there exists $\lambda_0 > 0$ such that $\lambda(t) \geq \lambda_0$ for all $t \geq t_0$. Then for every $(x_0, u_0) \in H \times H $ there exists a unique strong global solution $x: [t_0, +\infty) \mapsto H$ of the continuous dynamics \eqref{Syst} which satisfies the Cauchy initial conditions $x(t_0) = x_0$ and $\dot x(t_0) = u_0$.
\end{theorem}

\section{Energy function and rates of convergence for function values}

In this section we will define for the dynamical system \eqref{Syst} an energy function and investigate its dissipativity properties. These will play a crucial role in the derivation of rates of convergence for the Moreau envelope of $\Phi$ and the objective function itself. 

To shorten the calculations, we introduce the auxiliary function (see also \cite{ACFR})
 \begin{equation*}
w: [t_0, +\infty) \to \mathbb{R}, \quad w(t) \ = \ b(t) - \dot \beta(t) -\frac{\beta(t)}{t}.
\end{equation*}

For $z \in \argmin \Phi$ and 
\begin{equation}\label{c}
0 \leq c \leq \alpha - 1,
\end{equation}
consider the  energy function $E_c : [t_0,+\infty) \rightarrow [0,+\infty)$,
\begin{align*}
E_c(t) \ = \ &\big( t^2 w(t) + (\alpha-1-c) t \beta(t) \big) (\Phi_{\lambda(t)} (x(t)) - \Phi^*) + \frac{1}{2} \left \| c(x(t) - z) + t \dot x(t) + t \beta(t) \nabla \Phi_{\lambda(t)} (x(t)) \right \|^2 \\
&+ \ \frac{c(\alpha - 1 - c)}{2} \| x(t) - z \|^2.
\end{align*}

In the following theorem we formulate sufficient conditions that guarantee the decay of the energy of the the dynamical system \eqref{Syst}  and discuss  some of its consequences. 

\begin{theorem}\label{Th_2}

Suppose that $\alpha \geq 1$, $\lambda$ is nondecreasing on $[t_0,+\infty)$ and the following conditions
\begin{equation}\label{tbbeta}
b(t) > \dot \beta(t) + \frac{\beta(t)}{t} \ \mbox{for every} \ t \geq t_0
\end{equation}
and
\begin{equation}\label{As_1}
(\alpha - 3) w(t) - t \dot w(t) \geq 0 \ \mbox{for every} \ t \geq t_0
\end{equation}
are satisfied. Then, for a solution $x : [t_0,+\infty) \rightarrow H$ to  \eqref{Syst}, the following statements are true:

\begin{enumerate}[label=(\roman*)]

\item\label{i} $\dot E_c (t) \leq 0$ for every $t \geq t_0$;

\item\label{ii} $ \Phi_{\lambda(t)} (x(t)) - \Phi^* \ \leq \ \frac{E_{\alpha - 1}(t_0)}{t^2w(t)} $ for every $t \geq t_0$;

\item\label{iii} $ \int_{t_0}^{+\infty} \Big( t^2 w(t) \frac{\dot \lambda(t)}{2} + t^2 \beta(t) w(t) \Big) \| \nabla \Phi_{\lambda(t)} (x(t)) \|^2 dt \ < \ +\infty $;

\item\label{iv} $ \int_{t_0}^{+\infty} \big( (\alpha - 3)tw(t) - t^2 \dot w(t) \big) (\Phi_{\lambda(t)} (x(t)) - \Phi^*) dt \ < \ +\infty $.

Assuming moreover that $\alpha > 1$ and that 
\begin{equation}\label{As_2}
\mbox{there exists} \ \varepsilon \in (0, \alpha - 1) \ \mbox{such that} \ (\alpha - 3) w(t) - t \dot w(t) \ \geq \ \varepsilon b(t) \quad \forall t \geq t_0,
\end{equation}
it holds

\item\label{v} $ \int_{t_0}^{+\infty} t \| \dot x(t) \|^2 dt \ < \ +\infty $;

\item\label{vi} the trajectory $x$ is bounded and
\item\label{vii} $ \int_{t_0}^{+\infty} t b(t) (\Phi_{\lambda(t)} (x(t)) - \Phi^*) dt \ < \ +\infty $.
\end{enumerate}
\end{theorem}

\begin{proof}
For every $t \geq t_0$ we obtain
\begin{align*}
\dot E_c(t) \ = \ &\big( 2tw(t) + t^2 \dot w(t) + \beta(t) (\alpha-1-c) + (\alpha-1-c) t \dot \beta(t) \big) (\Phi_{\lambda(t)} (x(t)) - \Phi^*) \\ 
&+ \ \big( t^2 w(t) + \beta(t) t (\alpha-1-c) \big) \Big( \langle \nabla \Phi_{\lambda(t)} (x(t)), \dot x(t) \rangle - \frac{\dot \lambda(t)}{2} \| \nabla \Phi_{\lambda(t)} (x(t)) \|^2 \Big) \\ 
&+ \ \Big\langle c(x(t) - z) + t \dot x(t) + t \beta(t) \nabla \Phi_{\lambda(t)} (x(t)), (c + 1) \dot x(t) + t \ddot x(t) + t \beta(t) \frac{d}{dt} \left( \nabla \Phi_{\lambda(t)} (x(t)) \right) \\
&+ \ (\beta(t) + t \dot \beta(t)) \nabla \Phi_{\lambda(t)} (x(t)) \Big\rangle \ + \ c(\alpha - 1 - c) \langle x(t) - z, \dot x(t) \rangle,
\end{align*}
where we used that
\begin{equation}\label{formula}
\frac{d}{dt} \left( \Phi_{\lambda(t)}(x(t)) - \Phi^* \right) = \big \langle \nabla \Phi_{\lambda(t)}(x(t)), \dot x(t) \big \rangle - \frac{\dot \lambda(t)}{2} \| \nabla \Phi_{\lambda(t)} x(t) \|^2.
\end{equation}
Using (\ref{Syst}) to replace $ \ddot x(t) $, we may write the third summand in the formulation of $\dot E_c(t)$ for every $t \geq t_0$ as
\begin{align*}
&\Big\langle c(x(t) - z) + t \dot x(t) + t \beta(t) \nabla \Phi_{\lambda(t)} (x(t)), (c + 1 - \alpha) \dot x(t) + (\beta(t) + t \dot \beta(t) - tb(t)) \nabla \Phi_{\lambda(t)} (x(t)) \Big\rangle \\
= \ &c(c + 1 - \alpha) \langle x(t) - z, \dot x(t) \rangle + c \big( \beta(t) + t \dot \beta(t) - tb(t) \big) \big \langle x(t) - z, \nabla \Phi_{\lambda(t)} (x(t)) \big \rangle + (c + 1 - \alpha)t \| \dot x(t) \|^2 \\ 
&+ \ \big( \beta(t) + t \dot \beta(t) - tb(t) \big) t \big \langle \dot x(t), \nabla \Phi_{\lambda(t)}(x(t)) \big \rangle + t \beta(t) (c + 1 - \alpha) \big\langle \dot x(t), \nabla \Phi_{\lambda(t)} (x(t)) \big\rangle \\ 
&+ \ t \beta(t) (\beta(t) + t \dot \beta(t) - tb(t)) \| \nabla \Phi_{\lambda(t)} (x(t)) \|^2.
\end{align*}
Overall, since $\beta(t) + t \dot \beta(t) - t b(t) = -t w(t)$, we obtain for every $t \geq t_0$
\begin{align*}
\dot E_c(t) \ = \ &\big( 2tw(t) + t^2 \dot w(t) - (\beta(t) + t \dot \beta(t)) (c + 1 - \alpha) \big) (\Phi_{\lambda(t)} (x(t)) - \Phi^*) \\
&+ \ \big( t^2 w(t) - t \beta(t) (c + 1 - \alpha) \big) \Big( \langle \nabla \Phi_{\lambda(t)} (x(t)), \dot x(t) \rangle - \frac{\dot \lambda(t)}{2} \| \nabla \Phi_{\lambda(t)} (x(t)) \|^2 \Big) \\ 
&- \ c t w(t) \big \langle x(t) - z, \nabla \Phi_{\lambda(t)} (x(t)) \big \rangle + (c + 1 - \alpha)t \| \dot x(t) \|^2 - t^2 w(t) \langle \dot x(t), \nabla \Phi_{\lambda(t)}(x(t)) \rangle \\ 
&+ \ t \beta(t) (c + 1 - \alpha) \big\langle \dot x(t), \nabla \Phi_{\lambda(t)} (x(t)) \big\rangle - t^2 \beta(t) w(t) \| \nabla \Phi_{\lambda(t)} (x(t)) \|^2.
\end{align*}
Notice that the terms with $ \langle \nabla \Phi_{\lambda(t)} (x(t)), \dot x(t) \rangle $ cancel each other, thus, after simplification we obtain for every $t \geq t_0$
\begin{equation}\label{dE}
\begin{split}
\dot E_c(t) \ = \ &\big( 2tw(t) + t^2 \dot w(t) + (\beta(t) + t \dot \beta(t)) (\alpha -1 -c) \big) (\Phi_{\lambda(t)} (x(t)) - \Phi^*) \\
&- \ \big( t^2 w(t) + t \beta(t) (\alpha -1 -c) \big) \frac{\dot \lambda(t)}{2} \| \nabla \Phi_{\lambda(t)} (x(t)) \|^2 - c t w(t) \big \langle x(t) - z, \nabla \Phi_{\lambda(t)} (x(t)) \big \rangle \\ 
&- \ (\alpha-1-c)t \| \dot x(t) \|^2 - t^2 \beta(t) w(t) \| \nabla \Phi_{\lambda(t)} (x(t)) \|^2.
\end{split}
\end{equation}
Thanks to (\ref{tbbeta}), $w(t)$ is positive for every $t \geq t_0$, thus
\[
-c t w(t) \big \langle x(t) - z, \nabla \Phi_{\lambda(t)} (x(t)) \big \rangle \ \leq \ -c t w(t) (\Phi_{\lambda(t)}(x(t)) - \Phi^*),
\]
which leads to
\begin{equation}\label{LF_0}
\begin{split}
\dot E_c(t) \ \leq \ &\big( (2 - c)tw(t) + t^2 \dot w(t) + (\beta(t) + t \dot \beta(t)) (\alpha -1 -c) \big) \big( \Phi_{\lambda(t)} (x(t)) - \Phi^* \big) \\ 
&- \ \Bigg( \big( t^2 w(t) + t \beta(t) (\alpha - 1 - c) \big) \frac{\dot \lambda(t)}{2} + t^2 \beta(t) w(t) \Bigg) \| \nabla \Phi_{\lambda(t)} (x(t)) \|^2 - \ (\alpha - c - 1)t \| \dot x(t) \|^2.
\end{split}
\end{equation}
By \eqref{c} and the fact that $\lambda$ is nondecreasing, we deduce that
\[
\big( t^2 w(t) + t \beta(t) (\alpha -1 -c) \big) \frac{\dot \lambda(t)}{2} + t^2 \beta(t) w(t) \ \geq \ 0,
\]
so, we obtain for every $t \geq t_0$
\begin{equation}\label{LF}
\begin{split}
\dot E_c(t) \ \leq \ &\big( (2 - c)tw(t) + t^2 \dot w(t) + (\beta(t) + t \dot \beta(t)) (\alpha - 1 - c) \big) (\Phi_{\lambda(t)} (x(t)) - \Phi^*) \\ 
&- \ \Bigg( \big( t^2 w(t) + t \beta(t) (\alpha - c - 1) \big) \frac{\dot \lambda(t)}{2} + t^2 \beta(t) w(t) \Bigg) \| \nabla \Phi_{\lambda(t)} (x(t)) \|^2.
\end{split}
\end{equation}

Let us choose $c := \alpha - 1$. According to (\ref{As_1}) we obtain for the coefficient of $\Phi_{\lambda(t)} (x(t)) - \Phi^*$ in \eqref{LF}
\[
(2 - c)tw(t) + t^2 \dot w(t) + (\beta(t) + t \dot \beta(t)) (\alpha - 1 - c) \ = -t \big( (\alpha-3) w(t) - t \dot w(t) \big) \ \leq \ 0.
\]
Therefore, \eqref{LF} allows us to deduce 
\begin{align*}
\dot E_{\alpha - 1}(t) \ &= \ -\big( (\alpha-3)tw(t) - t^2 \dot w(t) \big) (\Phi_{\lambda(t)} (x(t)) - \Phi^*) - \Big( t^2 w(t) \frac{\dot \lambda(t)}{2} + t^2 \beta(t) w(t) \Big) \| \nabla \Phi_{\lambda(t)} (x(t)) \|^2 \\
&\leq \ 0. 
\end{align*}
We have just established that $E_{\alpha-1}$ is nonincreasing, which leads for every $t\geq t_0$ to 
\begin{align*}
E_{\alpha - 1}(t) \ &= \ t^2 w(t) (\Phi_{\lambda(t)} (x(t)) - \Phi^*) + \frac{1}{2} \left \| (\alpha - 1)(x(t) - z) + t \dot x(t) + t \beta(t) \nabla \Phi_{\lambda(t)} (x(t)) \right \|^2 \\
&\leq \ E_{\alpha - 1} (t_0).
\end{align*}
From here we obtain for every $t\geq t_0$
\begin{equation}\label{R}
\Phi_{\lambda(t)} (x(t)) - \Phi^* \ \leq \ \frac{E_{\alpha - 1}(t_0)}{t^2w(t)},
\end{equation}
which proves \ref{ii}. Moreover, by integration, we obtain
\begin{equation}\label{I}
\int_{t_0}^{+\infty} \left( t^2 w(t) \frac{\dot \lambda(t)}{2} + t^2 \beta(t) w(t) \right) \| \nabla \Phi_{\lambda(t)} (x(t)) \|^2 dt \ \leq \  E_{\alpha-1}(t_0) \ < \ +\infty 
\end{equation}
and
\begin{equation}\label{I_1}
\int_{t_0}^{+\infty} \big( (\alpha - 3)tw(t) - t^2 \dot w(t) \big) (\Phi_{\lambda(t)} (x(t)) - \Phi^*) dt \ \leq \ E_{\alpha - 1}(t_0) \ < \ +\infty,
\end{equation}
which are the claims \ref{iii} and \ref{iv}.

From now on we assume that $\alpha > 1$ and choose $c := \alpha - 1 - \varepsilon$, where $\varepsilon$ is given by \eqref{As_2}. In this setting, \eqref{LF_0} reads for every $t \geq t_0$,
\begin{align}\label{LF_1}
\dot E_{\alpha - 1 - \varepsilon}(t) \ \leq \ &\big( (3 - \alpha + \varepsilon)tw(t) + t^2 \dot w(t) + \varepsilon (\beta(t) + t \dot \beta(t)) \big) (\Phi_{\lambda(t)} (x(t)) - \Phi^*) \nonumber \\ 
&- \ \Bigg( \big( t^2 w(t) + \varepsilon t \beta(t) \big) \frac{\dot \lambda(t)}{2} + t^2 \beta(t) w(t) \Bigg) \| \nabla \Phi_{\lambda(t)} (x(t)) \|^2 - \varepsilon t \| \dot x(t) \|^2 \nonumber \\
\ = \ & - t \big( (\alpha-3)w(t) - t \dot w(t) - \varepsilon b(t) \big) (\Phi_{\lambda(t)} (x(t)) - \Phi^*) \nonumber \\ 
&- \ \Bigg( \big( t^2 w(t) + \varepsilon t \beta(t) \big) \frac{\dot \lambda(t)}{2} + t^2 \beta(t) w(t) \Bigg) \| \nabla \Phi_{\lambda(t)} (x(t)) \|^2 - \varepsilon t \| \dot x(t) \|^2.
\end{align}
So, under the condition \eqref{As_2}, $\dot E_{\alpha - 1 - \varepsilon}(t) \ \leq \ 0$ for every $t \geq t_0$. Integrating \eqref{LF_1} we obtain
\begin{equation}\label{I_4}
\int_{t_0}^{+\infty} t \| \dot x(t) \|^2 dt \ < \ +\infty,
\end{equation}
which gives the claim \ref{v}. From the fact that the energy function 
\begin{align*}
E_{\alpha - 1 - \varepsilon}(t) \ = \ &\big( t^2 w(t) + \varepsilon t \beta(t) \big) (\Phi_{\lambda(t)} (x(t)) - \Phi^*) \\ 
&+ \ \frac{1}{2} \| (\alpha - 1 - \varepsilon)(x(t) - z) + t \dot x(t) + t \beta(t) \nabla \Phi_{\lambda(t)} (x(t)) \|^2 + \frac{(\alpha - 1 - \varepsilon)\varepsilon}{2} \| x(t) - z \|^2
\end{align*}
is bounded from above and it is nonnegative on $[t_0, +\infty)$, it follows that the trajectory $x$ is bounded, which is item \ref{vi}. Finally, from  \eqref{As_2} and \eqref{I_1} we deduce the claim \ref{vii}
\begin{equation}\label{I_2}
\int_{t_0}^{+\infty} \varepsilon t b(t) (\Phi_{\lambda(t)} (x(t)) - \Phi^*) dt \ \leq \ \int_{t_0}^{+\infty} \big( (\alpha - 3)tw(t) - t^2 \dot w(t) \big) (\Phi_{\lambda(t)} (x(t)) - \Phi^*) dt \ < \ +\infty,
\end{equation}
which finishes the proof.
\end{proof}

The following auxiliary result will be needed later. 
\begin{lemma}\label{lemma3}
Suppose that  $\alpha >1$ and  \eqref{As_2} holds, that $\lambda$ and $\beta$ are nondecreasing on $[t_0, +\infty)$, and that \eqref{tbbeta} holds. Then, for a solution $x : [t_0,+\infty) \rightarrow H$ to  \eqref{Syst}, it holds
\begin{equation}\label{I_3}
\int_{t_0}^{+\infty} t w(t) \langle \nabla \Phi_{\lambda(t)} (x(t)), x(t) - z \rangle dt \ < \ +\infty.
\end{equation}
\end{lemma}

\begin{proof}
Recall that according to \eqref{dE} we have for every $t \geq t_0$
\begin{align*}
\dot E_c(t) \ = \ &\big( 2tw(t) + t^2 \dot w(t) + (\beta(t) + t \dot \beta(t)) (\alpha - 1 - c) \big) (\Phi_{\lambda(t)} (x(t)) - \Phi^*) \\ 
&- \ \big( t^2 w(t) + t \beta(t) (\alpha-1-c) \big) \frac{\dot \lambda(t)}{2} \| \nabla \Phi_{\lambda(t)} (x(t)) \|^2 \\ 
&- \ c t w(t) \big \langle x(t) - z, \nabla \Phi_{\lambda(t)} (x(t)) \big \rangle - (\alpha-1-ct) \| \dot x(t) \|^2 -  t^2 \beta(t) w(t) \| \nabla \Phi_{\lambda(t)} (x(t)) \|^2.
\end{align*}
We choose again $c: = \alpha - 1$ and split the term $(\alpha - 1)t w(t) \langle x(t) - z, \nabla \Phi_{\lambda(t)} (x(t)) \rangle$ into the sum of two expressed in terms of $\varepsilon$  given by  \eqref{As_2}. For every $t \geq t_0$, we have
\begin{align*}
\dot E_{\alpha - 1}(t) \ \leq \ &\big( 2tw(t) + t^2 \dot w(t) \big) (\Phi_{\lambda(t)} (x(t)) - \Phi^*) \\
&- \ ( \alpha - 1 - \varepsilon) t w(t) \big \langle x(t) - z, \nabla \Phi_{\lambda(t)} (x(t)) \rangle - \varepsilon t w(t) \big \langle x(t) - z, \nabla \Phi_{\lambda(t)} (x(t)) \big \rangle.
\end{align*}
By applying the convex subdifferential inequality we obtain for every $t \geq t_0$
\begin{align}\label{dE_1}
\dot E_{\alpha - 1}(t) \ \leq \ &\big( 2tw(t) + t^2 \dot w(t) - (\alpha - 1 - \varepsilon) t w(t) \big) (\Phi_{\lambda(t)} (x(t)) - \Phi^*) - \ \varepsilon t w(t) \big \langle x(t) - z, \nabla \Phi_{\lambda(t)} (x(t)) \big \rangle \nonumber\\
\ =  \ &\big( t^2 \dot w(t) - \left( \alpha - 3 - \varepsilon \right) t w(t) \big) (\Phi_{\lambda(t)} (x(t)) - \Phi^*) - \ \varepsilon t w(t) \big \langle x(t) - z, \nabla \Phi_{\lambda(t)} (x(t)) \big \rangle.
\end{align}

Since $\beta$ is nondecreasing, for every $t \geq t_0$ it holds
\[
b(t) = w(t) + \dot \beta(t) + \frac{\beta(t)}{t} \geq w(t),
\]
thus, \eqref{As_2} leads to $ t^2 \dot w(t) - (\alpha - 3 -\varepsilon) t w(t) \leq 0$. Consequently, we obtain from \eqref{dE_1} by integration
\[
\int_{t_0}^{+\infty} t w(t) \langle \nabla \Phi_{\lambda(t)} (x(t)), x(t) - z \rangle dt \ \leq \ E_{\alpha - 1}(t_0) \ < \ +\infty.
\]\end{proof}

Now we are in position to improve the convergence rates which we obtained previously in \eqref{R} and to derive from here convergence rates for $\Phi$.

\begin{theorem}\label{CFVG}
Suppose that  $\alpha >1$ and  \eqref{As_2} holds, that $\lambda$ and $\beta$ are nondecreasing on $[t_0, +\infty)$, and that \eqref{tbbeta} holds. Assume in addition
\begin{equation}\label{As_4}
\int_{t_0}^{+\infty} \left[ \frac{ \big(\dot \lambda(t)\big)^2 t^3 \beta^2(t)}{\lambda^4(t)} - \frac{\dot \lambda(t) t^2 b(t)}{2 \lambda^2(t)} \right]_+ dt < +\infty,
\end{equation}
where $[\cdot]_+$ denotes the positive part of the expression inside the brackets, and that there exists $C>0$ such that
\begin{equation}\label{As}
\frac{d}{dt} \left( t^2 b(t) \right) \ \leq \ C t b(t) \ \mbox{for every} \ t \geq t_0.
\end{equation}
Then, for a solution $x : [t_0,+\infty) \rightarrow H$ to  \eqref{Syst}, it holds
\begin{equation}\label{rates_1}
\Phi_{\lambda(t)}(x(t)) - \Phi^* = o \left( \frac{1}{t^2 b(t)} \right) \text{ and } \| \dot x(t) \| = o \left( \frac{1}{t} \right) \text{ as } t \to +\infty.
\end{equation}
Moreover, 
\begin{equation}\label{rates_111}
\| \nabla \Phi_{\lambda(t)} (x(t)) \| \ = \ o \left( \frac{1}{t \sqrt{b(t) \lambda(t)}} \right) \text{ as } t \to +\infty,
\end{equation}
and
\begin{equation}\label{rates_11}
\Phi(\prox\nolimits_{\lambda(t) \Phi} (x(t))) - \Phi^* = o \left( \frac{1}{t^2 b(t)} \right) \text{ and } \| \prox\nolimits_{\lambda(t) \Phi} (x(t)) - x(t) \| \ = \ o \left( \frac{\sqrt{\lambda(t)}}{t \sqrt{b(t)}} \right) \text{ as } t \to +\infty.
\end{equation}
\end{theorem}

\begin{proof}
First we notice that for every $t \geq t_0$ it holds
\begin{align*} 
\left \langle \frac{d}{dt} \left( \nabla \Phi_{\lambda(t)}(x(t)) \right), \dot x(t) \right \rangle \ &= \ \left \langle \lim_{h \to 0} \frac{\nabla \Phi_{\lambda(t + h)} (x(t + h)) - \nabla \Phi_{\lambda(t)} (x(t))}{h}, \dot x(t) \right \rangle \\
&= \ \left \langle \lim_{h \to 0} \frac{\nabla \Phi_{\lambda(t + h)} (x(t + h)) - \nabla \Phi_{\lambda(t + h)} (x(t))}{h}, \dot x(t) \right \rangle \\ &\ \ \ \ + \ \left \langle \lim_{h \to 0} \frac{\nabla \Phi_{\lambda(t + h)} (x(t)) - \nabla \Phi_{\lambda(t)} (x(t))}{h}, \dot x(t) \right \rangle.
\end{align*}
For every $h > 0$, by the monotonicity of the gradient of a convex function, we have 
\[
\left \langle \frac{\nabla \Phi_{\lambda(t + h)} (x(t + h)) - \nabla \Phi_{\lambda(t + h)} (x(t))}{h}, \frac{x(t + h) - x(t)}{h} \right \rangle \ \geq \ 0,
\]
so letting $h$ tend to zero we obtain
\[
\left \langle \lim_{h \to 0} \frac{\nabla \Phi_{\lambda(t + h)} (x(t + h)) - \nabla \Phi_{\lambda(t + h)} (x(t))}{h}, \dot x(t) \right \rangle \ \geq \ 0.
\]
Consequently, for every $t \geq t_0$ it holds
\begin{align*}
&\left \langle \frac{d}{dt} \left( \nabla \Phi_{\lambda(t)}(x(t)) \right), \dot x(t) \right \rangle \\ 
\geq & \ \left \langle \lim_{h \to 0} \frac{\nabla \Phi_{\lambda(t + h)} (x(t)) - \nabla \Phi_{\lambda(t)} (x(t))}{h}, \dot x(t) \right \rangle \\
= \ &\lim_{h \to 0} \left \langle \frac{ (\lambda(t + h)) \prox_{\lambda(t) \Phi} (x(t)) - \lambda(t) \prox_{(\lambda(t + h)) \Phi} (x(t)) - \big( \lambda(t + h) - \lambda(t) \big) x(t) }{\lambda(t) \lambda(t + h)h}, \dot x(t) \right \rangle \\
= \ &\lim_{h \to 0} \left \langle \frac{\big( \lambda(t + h) - \lambda(t) \big) \big( \prox_{\lambda(t) \Phi} (x(t)) - x(t) \big)}{\lambda(t) \lambda(t + h) h}, \dot x(t) \right \rangle \\
&- \ \lim_{h \to 0} \left \langle \frac{\prox_{(\lambda(t + h)) \Phi} (x(t)) - \prox_{\lambda(t) \Phi} (x(t))}{\lambda(t+h)h}, \dot x(t) \right \rangle \\
\geq \ &\frac{\dot \lambda(t)}{\lambda^2(t)} \left \langle \prox\nolimits_{\lambda(t) \Phi} (x(t)) - x(t), \dot x(t) \right \rangle \ - \ \lim_{h \to 0} \frac{(\lambda(t + h) - \lambda(t)) \| \nabla \Phi_{\lambda(t)} (x(t)) \| \| \dot x(t) \| }{\lambda(t + h) h} \\
= \ &\frac{\dot \lambda(t)}{\lambda^2(t)} \left \langle \prox\nolimits_{\lambda(t) \Phi} (x(t)) - x(t), \dot x(t) \right \rangle \ - \ \frac{\dot \lambda(t) \| \nabla \Phi_{\lambda(t)} (x(t)) \| \| \dot x(t) \|}{\lambda(t)} \\
 = \ & -\frac{\dot \lambda(t)}{\lambda(t)} \left \langle \nabla \Phi_{\lambda(t)} (x(t)), \dot x(t) \right \rangle - \ \frac{\dot \lambda(t) \| \nabla \Phi_{\lambda(t)} (x(t)) \| \| \dot x(t) \|}{\lambda(t)} \ \geq \ - \frac{2 \dot \lambda(t) \| \nabla \Phi_{\lambda(t)} (x(t)) \| \| \dot x(t) \|}{\lambda(t)},
\end{align*}
where we used \eqref{Morprox}, \eqref{proxineq} and the Cauchy-Schwarz inequality. Now we multiply (\ref{Syst}) by $t^2 \dot x(t)$ to deduce, by using the inequality above and \eqref{formula}, for every $t \geq t_0$
\begin{align*}
0 \ = & \ t^2 \langle \ddot x(t), \dot x(t) \rangle + \alpha t \| \dot x(t) \|^2 + t^2 \beta(t) \left \langle \frac{d}{dt} \left( \nabla \Phi_{\lambda(t)}(x(t)) \right), \dot x(t) \right \rangle + t^2 b(t) \big \langle \nabla \Phi_{\lambda(t)}(x(t)), \dot x(t) \big \rangle \\
\geq & \ t^2 \frac{d}{dt} \left( \frac{1}{2} \| \dot x(t) \|^2 \right) + \alpha t \| \dot x(t) \|^2 + t^2 b(t) \frac{d}{dt} \left( \Phi_{\lambda(t)}(x(t)) - \Phi^* \right) + \frac{\dot \lambda(t) t^2 b(t)}{2} \| \nabla \Phi_{\lambda(t)} x(t) \|^2 \\ 
& - \ \frac{2 t^2 \beta(t) \dot \lambda(t)}{\lambda(t)} \| \nabla \Phi_{\lambda(t)} (x(t)) \| \| \dot x(t) \| \\ 
 \geq & \ \frac{d}{dt} \left( \frac{t^2}{2} \| \dot x(t) \|^2 + t^2 b(t) (\Phi_{\lambda(t)}(x(t)) - \Phi^*) \right) + (\alpha - 1) t \| \dot x(t) \|^2 - \big( \Phi_{\lambda(t)}(x(t)) - \Phi^* \big) \frac{d}{dt} \left( t^2 b(t) \right)\\ 
& - \ \left\{ \left[ \left( \frac{\dot \lambda(t)}{\lambda(t)} \right)^2 t^3 \beta^2(t) - \frac{\dot \lambda(t) t^2 b(t)}{2} \right] \| \nabla \Phi_{\lambda(t)} (x(t)) \|^2 + t \| \dot x(t) \|^2 \right\}.
\end{align*}
Using \eqref{As} we obtain for every $t \geq t_0$
\begin{align*}
\frac{d}{dt} \left( \frac{t^2}{2} \| \dot x(t) \|^2 + t^2 b(t) (\Phi_{\lambda(t)}(x(t)) - \Phi^*) \right) \leq \ &[2 - \alpha]_+ t \| \dot x(t) \|^2 + \big( \Phi_{\lambda(t)}(x(t)) - \Phi^* \big) C t b(t) \\
&+ \ \left[ \left( \frac{\dot \lambda(t)}{\lambda(t)} \right)^2 t^3 \beta^2(t) - \frac{\dot \lambda(t) t^2 b(t)}{2} \right]_+ \| \nabla \Phi_{\lambda(t)} (x(t)) \|^2.
\end{align*}
Next we show the integrability of the right-hand side of the expression above. The first term is integrable according to Theorem \ref{Th_2} (v) and the second one is integrable according to Theorem \ref{Th_2} (vii). Further, since 
\[
\|  \nabla \Phi_{\lambda(t)}(x(t)) -  \nabla \Phi_{\lambda(t)}(z) \| \ \leq \ \frac{1}{\lambda(t)} \| x(t) - z \| \ \forall t \geq t_0,
\]
and taking into the account the boundedness of the trajectory $x$ established in Theorem \ref{Th_2} (vi) and that $z \in \argmin \Phi$, we deduce
\begin{equation*}
\| \nabla \Phi_{\lambda(t)}(x(t)) \| \ = \ O \left( \frac{1}{\lambda(t)} \right) \text{ as } t \to +\infty.
\end{equation*}
So, under the assumption \eqref{As_4}, we obtain that there exists $\widetilde C > 0$ such that for every $t \geq t_0$
\begin{align*}
\int_{t_0}^t \left[ \left( \frac{\dot \lambda(s)}{\lambda(s)} \right)^2 s^3 \beta^2(s) - \frac{\dot \lambda(s) s^2 b(s)}{2} \right]_+ \| \nabla \Phi_{\lambda(s)} (x(s)) \|^2 ds &\leq \widetilde C \int_{t_0}^t \left[ \frac{\big(\dot \lambda(s)\big)^2 s^3 \beta^2(s)}{\lambda^4(s)} - \frac{\dot \lambda(s) s^2 b(s)}{2 \lambda^2(s)} \right]_+ ds \\
&< \ +\infty.
\end{align*}
Applying Lemma \ref{l}  in the Appendix, we conclude that the following limit 
\[
L := \lim_{t \to +\infty} \left( \frac{t^2}{2} \| \dot x(t) \|^2 + t^2 b(t) (\Phi_{\lambda(t)}(x(t)) - \Phi^*) \right) \geq 0
\]
exists. We will show that $L = 0$. Supposing that $L >0$, we deduce that there exists $t^* \geq t_0$ such that for every $t \geq t^*$
\[
\frac{t}{2} \| \dot x(t) \|^2 + t b(t) (\Phi_{\lambda(t)}(x(t)) - \Phi^*) \geq \frac{L}{2t}.
\]
Integrating the last inequality on $[t^*, +\infty)$, we arrive at the contradiction with the integrability of the left-hand side as proved in Theorem \ref{Th_2} (v) and (vii). Therefore, $L = 0$ and we obtain
\[
\Phi_{\lambda(t)}(x(t)) - \Phi^* = o \left( \frac{1}{t^2 b(t)} \right) \text{ and } \| \dot x(t) \| = o \left( \frac{1}{t} \right) \text{ as } t \to +\infty.
\]
Using the definition of the proximal mapping, we derive
\begin{equation}\label{proxPhit}
\Phi_{\lambda(t)}(x(t)) - \Phi^* = \ \Phi(\prox\nolimits_{\lambda(t) \Phi}(x(t))) - \Phi^* + \frac{1}{2\lambda(t)} \| \prox\nolimits_{\lambda(t) \Phi} (x(t)) - x(t) \|^2 \quad \forall t \geq t_0,
\end{equation}
which yields
\[
\Phi(\prox\nolimits_{\lambda(t) \Phi}(x(t))) - \Phi^* = o \left( \frac{1}{t^2 b(t)} \right) \text{ and } \|\prox\nolimits_{\lambda(t) \Phi}(x(t)) - x(t) \| \ = \ o \left( \frac{\sqrt{\lambda(t)}}{t \sqrt{b(t)}} \right) \text{ as } t \to +\infty.
\]
According to \eqref{Morprox} we obtain from here
\[
\| \nabla \Phi_{\lambda(t)} (x(t)) \| \ = \ o \left( \frac{1}{t \sqrt{b(t) \lambda(t)}} \right) \text{ as } t \to +\infty.
\]
\end{proof}

\section{Convergence of the trajectories}

In this section we will investigate the weak convergence of the trajectory $x$ to a minimizer of $\Phi$.

\begin{theorem}\label{CT}
Suppose that  $\alpha >1$, \eqref{tbbeta} and \eqref{As_2} hold and that $\lambda$ and $\beta$ are nondecreasing on $[t_0, +\infty)$. Assume in addition that
\begin{equation}\label{As_3}
\lim_{t \to +\infty} \frac{\beta(t)}{t w(t)} \ = \ 0
\end{equation}
and
\begin{equation}\label{As_5}
\sup_{t \geq t_0} \frac{\lambda(t)}{t} \ < \ +\infty.
\end{equation}
If $x : [t_0,+\infty) \rightarrow H$ is a solution to \eqref{Syst}, then $x(t)$ converges weakly to a minimizer of $\Phi$ 
as $t \to +\infty$.
\end{theorem}

\begin{proof}
Let $z \in \argmin \Phi$. Previously, in Theorem \ref{Th_2}, we established the existence of the limit of $E_c(t)$ as $t\to + \infty$ for $c = \alpha - 1$ and $c = \alpha - 1 - \varepsilon$, where $\varepsilon \in (0, \alpha-1)$ is given by \eqref{As_2}. Thus, computing the difference
\begin{align*}
E_{\alpha - 1 - \varepsilon}(t) - E_{\alpha - 1}(t) \ = \ &\varepsilon t \beta (t)(\Phi_{\lambda(t)}(x(t)) - \Phi^*) + \frac{\varepsilon (\alpha - 1)}{2} \|x(t) - z \|^2 \\
&- \varepsilon \langle (\alpha - 1) (x(t) - z) + t (\dot x(t) + \beta(t) \nabla \Phi_{\lambda(t)}(x(t)), x(t) - z \rangle\\
= \ &\varepsilon t \beta (t)(\Phi_{\lambda(t)}(x(t)) - \Phi^*) - \frac{\varepsilon (\alpha -1)}{2} \| x(t) - z \|^2\\
\ & - \ \varepsilon t \langle \dot x(t) + \beta(t) \nabla \Phi_{\lambda(t)}(x(t)), x(t) - z \rangle,
\end{align*}
we deduce that the limit of the right-hand side exists. Thanks to \eqref{R}, we derive for every $t \geq t_0$
\[
t \beta(t) (\Phi_{\lambda(t)}(x(t)) - \Phi^*) \ \leq \ t \beta(t) \frac{E_{\alpha - 1}(t_0)}{t^2 w(t)} \ = \ E_{\alpha - 1}(t_0) \frac{\beta(t)}{t w(t)}
\]
and from here, based on the assumption \eqref{As_3}, we obtain
\begin{align}\label{limitPhi}
\lim_{t \to +\infty} t \beta(t) (\Phi_{\lambda(t)}(x(t)) - \Phi^*) = 0.
\end{align}
Hereby, we derived that the limit of the quantity
\[
p(t) := \frac{\alpha - 1}{2} \| x(t) - z \|^2 + t \langle \dot{x}(t), x(t) - z \rangle + t \beta(t) \langle \nabla \Phi_{\lambda(t)}(x(t)), x(t) - z \rangle
\]
exists as $t \to +\infty$. Now we are ready to prove the existence of the limit of $\| x(t) - z \|$ as $t \to +\infty$. Denote
\[
q(t) := \frac{\alpha - 1}{2} \| x(t) - z \|^2 + (\alpha - 1)\int_{t_0}^t \beta(s) \langle \nabla \Phi_{\lambda(s)}(x(s)), x(s) - z \rangle ds \ \forall t \geq t_0.
\]
For every $t \geq t_0$, it holds that
\[
p(t )= q(t) + \frac{t}{\alpha - 1} \dot q(t) - (\alpha - 1)\int_{t_0}^t \beta(s) \langle \nabla \Phi_{\lambda(s)}(x(s)), x(s) - z \rangle ds,
\]
since 
\[
\dot q(t) = (\alpha - 1) \langle x(t) - z, \dot x(t) \rangle + \big( \alpha - 1 \big) \big( \beta(s) \langle \nabla \Phi_{\lambda(s)}(x(s)), x(s) - z \rangle \big)
\]
and
\begin{align*}
q(t) + \frac{t}{\alpha - 1} \dot q(t) \ = \ &\frac{\alpha - 1}{2} \| x(t) - z \|^2 + (\alpha - 1)\int_{t_0}^t \beta(s) \langle \nabla \Phi_{\lambda(s)}(x(s)), x(s) - z \rangle ds \\
&+ \ t \langle x(t) - z, \dot x(t) \rangle + t \big( \beta(s) \langle \nabla \Phi_{\lambda(s)}(x(s)), x(s) - z \rangle \big).
\end{align*}
By Lemma \ref{lemma3} we established that $\int_{t_0}^{+\infty} s w(s) \langle \nabla \Phi_{\lambda(s)}(x(s)), x(s) - z \rangle ds < +\infty$. In turn, \eqref{As_3} yields that
\begin{equation}\label{limit}
\lim_{t \to +\infty}  \int_{t_0}^t \beta(s) \langle \nabla \Phi_{\lambda(s)}(x(s)), x(s) - z \rangle ds \text{ exists }.
\end{equation}
Finally,
\[
\lim_{t \to +\infty} \left( q(t) + \frac{t}{\alpha - 1} \dot q(t) \right) \text{ also exists. }
\]
Applying now Lemma \ref{A} in the Appendix, we immediately get the existence of the limit of $q(t)$ as $t \to +\infty$. By the definition of $q$ and \eqref{limit} we establish the first statement of the Opial's Lemma (see Lemma \ref{O} in the Appendix), namely, that, for any $z \in \argmin \Phi$  
\[
\lim_{t \to +\infty}  \| x(t) - z \| \; \mbox{ exists}.
\]
To establish the second term of the Opial's Lemma, first note that from \eqref{proxPhit} and \eqref{limitPhi} we have, by denoting $\xi(t):=\prox_{\lambda(t)\Phi}(x(t))$, $\lim_{t \rightarrow +\infty} t\beta(t)(\Phi(\xi(t)) - \Phi^*) = 0$ and $\lim_{t \rightarrow +\infty} \frac{t\beta(t)}{\lambda(t)} \| \xi(t) - x(t) \|^2 = 0$. Using that $\beta$ is nondecreasing and assumption \eqref{As_5}, we deduce
$$ \lim_{t \to +\infty} \Phi(\xi(t)) \ = \ \Phi^* \quad \mbox{and} \quad \lim_{t \to +\infty} \| \xi(t) - x(t) \|  =0.$$
Considering a sequence $ \{ t_k \}_{k \in \mathbb N} $ such that $ \{x(t_k)\}_{k \in \mathbb N} $ converges weakly to an element $ z \in H $  as $ k \to +\infty $, we notice that  $\{\xi(t_k)\}_{k \in \mathbb N} $ converges weakly to $z$ as $ k \to +\infty $. Now, the function $ \Phi $ being convex and lower semicontinuous in the weak topology, allows us to write
\[
\Phi(z) \ \leq \ \liminf_{k \to +\infty} \Phi(\xi(t_k)) \ = \ \lim_{t \to +\infty} \Phi(\xi(t)) \ = \ \Phi^*. 
\]
Hence, $ z \in \argmin \Phi$, and the second statement of the Opial's Lemma is shown. This gives the weak convergence of the trajectory $x(t)$ to a minimizer of $\Phi$  as $t \rightarrow +\infty$.
\end{proof}

\begin{remark}\label{remark1}\rm 
In the hypotheses of Theorem \ref{CFVG}, in order to obtain the convergence of the trajectories, besides \eqref{As_5} it is enough to assume that 
$$\sup_{t \geq t_0}  \frac{\beta(t)}{t w(t)} < +\infty$$
in order to guarantee \eqref{limit}. Indeed, in this case \eqref{limitPhi} follows from the conclusion of Theorem \ref{CFVG}
$$\lim_{t \rightarrow +\infty} t \beta(t) (\Phi_{\lambda(t)}(x(t)) - \Phi^*)  \leq \lim_{t \rightarrow +\infty} t^2 b(t) (\Phi_{\lambda(t)}(x(t)) - \Phi^*) \frac{\beta(t)}{t w(t)} =0.$$
\end{remark}

\section{Polynomial choices for the system parameter functions}

According to the previous two sections, in order to guarantee both the fast convergence rates in Theorem \ref{CFVG} and the convergence of the trajectory to a minimizer of $\Phi$ in Theorem \ref{CT}, by taking also into account Remark \ref{remark1}, it is enough to make the following assumptions on the system parameter functions

\begin{enumerate}[label=(\Roman*)]

\item\label{0_I} $\alpha >1$ and there exists $\varepsilon \in (0, \alpha - 1) $ such that $ (\alpha - 3) w(t) - t \dot w(t) \geq \varepsilon b(t) $ for every $t \geq t_0$;

\item\label{II} $\beta$ and $\lambda$ are nondecreasing on $[t_0, +\infty)$;

\item\label{III} $b(t) > \dot \beta(t) + \frac{\beta(t)}{t} $ for every $t \geq t_0$;

\item\label{IV} $ \int_{t_0}^{+\infty} \left[ \frac{\beta^2(t) (\dot \lambda(t))^2 t^3}{\lambda^4(t)} - \frac{\dot \lambda(t) t^2 b(t)}{2 \lambda^2(t)} \right]_+ dt < +\infty $;

\item\label{V} there exists $C > 0 $ such that $ \frac{d}{dt} \left( t^2 b(t) \right) \ \leq \ C t b(t) $ for every $t \geq t_0$;

\item\label{VI} $ \sup_{t \geq t_0}  \frac{\beta(t)}{t w(t)} < +\infty$;

\item\label{VII} $ \sup_{t \geq t_0} \frac{\lambda(t)}{t} \ < \ +\infty$.
\end{enumerate}

In this section we will investigate the fulfillment of these conditions for
$$b(t) = b t^n, \quad \beta(t) = \beta t^m \quad \mbox{and} \quad \lambda(t) = \lambda t^l,$$ 
where $n,m,l \in \mathbb{R}$, $b, \lambda >0$ and $\beta \geq 0$.

For this choice of $b$, condition \ref{V} is fulfilled.

We assume first that $\beta =0$. Then the conditions \ref{III}, \ref{IV} and \ref{VI} are fulfilled, while the conditions \ref{II} and \ref{VI} are nothing else than $0 \leq l \leq 1$. Condition \ref{0_I} asks for $\alpha >1$ and for the existence of $\varepsilon \in (0, \alpha - 1)$ such that for every $t \geq t_0$
$$(\alpha - 3 - n - \varepsilon)bt^{n} \geq 0$$
or, equivalently, $\alpha - 3 - n \geq \varepsilon$. To this end it is enough to have that $\alpha-3 >n$.

In case $\beta >0$, conditions \ref{II} and \ref{VII} are nothing else than $m \geq 0$ and $0 \leq l \leq 1$. Condition \ref{III} reads for every $t \geq t_0$
\[
b t^n > m \beta t^{m-1} + \beta t^{m-1} = (m + 1) \beta t^{m-1},
\]
or, equivalently,
\[
t^{n - m + 1} > \frac{(m + 1) \beta}{b}.
\]
From here we get
\[
0 \leq m \leq n + 1.
\]
and $b > (m + 1) \beta t_0^{m - 1 - n}$. 

Condition \ref{VI}  requires that
\[
\sup_{t \geq t_0 }\frac{\beta t^m}{t (b t^n - \beta m t^{m-1} - \beta t^{m-1})} = \sup_{t \geq t_0} \frac{\beta t^m}{b t^{n+1} - \beta t^m (m + 1)} < +\infty
\]
and it is obviously fulfilled.

Condition \ref{0_I} asks for $\alpha >1$ and for the existence of $\varepsilon \in (0, \alpha - 1)$ such that for every $t \geq t_0$
\[
(\alpha - 3) (b t^n - \beta m t^{m-1} - \beta t^{m-1}) - t (b n t^{n-1} -\beta m (m - 1) t^{m-2} - \beta (m - 1) t^{m-2}) \geq \varepsilon b t^n.
\]
After simplification we obtain that for every $t \geq t_0$
\[
(\alpha - 3 - n - \varepsilon)b t^n + \beta (m + 1) (m + 2 -\alpha) t^{m-1} \geq 0
\]
or, equivalently,
\[
(\alpha - 3 - n - \varepsilon)b t^{n - m + 1} \geq \beta (m + 1) (\alpha - m - 2).
\]
On the one hand we have $m = n + 1$ and $(\alpha - 3 -n) \left( 1 - \frac{\beta(n + 2)}{b} \right) > \varepsilon$, which requires that $\alpha - 3 - n > 0$. On the other hand, we have $m < n + 1$, which also requires that $\alpha - 3 - n > 0$.

Consequently,  we have to assume that
\[
\alpha - 3 > n \quad \mbox{and} \quad b > \frac{\beta (m + 1) (\alpha - m - 2)}{(\alpha - 3 - n) t_0^{n - m + 1}}.
\]
In this case, there will be always an $\varepsilon \in (0, \alpha - 1)$ such that $\alpha - 3 - n - \varepsilon > 0$ and
\[
b > \frac{\beta (m + 1) (\alpha - m - 2)}{(\alpha - 3 - n - \varepsilon) t_0^{n - m + 1}} > \frac{\beta (m + 1) (\alpha - m - 2)}{(\alpha - 3 - n) t_0^{n - m + 1}},
\]
in other words, which satisfies condition \ref{0_I}.

Finally, let us  have a closer look at condition \ref{V}. This reads as
\[
\int_{t_0}^{+\infty} \left[ \frac{\beta^2 t^{2m} (l \lambda)^2 t^{2l - 2} t^3}{\lambda^4 t^{4l}} - \frac{l \lambda t^{l-1} t^2 b t^n}{2 \lambda^2 t^{2l}} \right]_+ dt < +\infty
\]
or, equivalently,
\[
\int_{t_0}^{+\infty} \left[ \left( \frac{\beta l }{\lambda} \right)^2 t^{2m -2l + 1} - \frac{l b}{2 \lambda} t^{n - l + 1} \right]_+ dt = \int_{t_0}^{+\infty} \left[ \left( \frac{\beta l }{\lambda} \right)^2 - \frac{l b}{2 \lambda} t^{n + l - 2m} \right]_+ t^{2m -2l + 1}  dt < +\infty.
\]

1. In case
\[
n + l > 2m,
\]
 there exists $t_1 \geq t_0$ such that for every $t \geq t_1$
\[
\left( \frac{\beta l }{\lambda} \right)^2 - \frac{l b}{2 \lambda} t^{n + l - 2m} \leq 0.
\]
Therefore, we obtain
\begin{align*}
&\int_{t_0}^{+\infty} \left[ \left( \frac{\beta l }{\lambda} \right)^2 - \frac{l b}{2 \lambda} t^{n + l - 2m} \right]_+ t^{2m -2l + 1} dt \\ 
= & \int_{t_0}^{t_1} \left[ \left( \frac{\beta l }{\lambda} \right)^2 - \frac{l b}{2 \lambda} t^{n + l - 2m} \right]_+ t^{2m -2l + 1} dt + \ \int_{t_1}^{+\infty} \left[ \left( \frac{\beta l }{\lambda} \right)^2 - \frac{l b}{2 \lambda} t^{n + l - 2m} \right]_+ t^{2m -2l + 1} dt \\
= & \int_{t_0}^{t_1} \left[ \left( \frac{\beta l }{\lambda} \right)^2 - \frac{l b}{2 \lambda} t^{n + l - 2m} \right]_+ t^{2m -2l + 1} dt 
< +\infty,
\end{align*}
thus \ref{V} is fulfilled.

2. In case 
\begin{equation*}
n + l < 2m,
\end{equation*}
there exist $\delta >0$ and $t_2 \geq t_0$ such that for all $t \geq t_2$
\[
\left( \frac{\beta l }{\lambda} \right)^2 - \frac{l b}{2 \lambda} t^{n + l - 2m} > \delta.
\]
Taking into account that $2m-2l+1 >n + 1 - l \geq -1$, we have
\begin{align*}
& \int_{t_0}^{+\infty} \left[ \left( \frac{\beta l }{\lambda} \right)^2 - \frac{l b}{2 \lambda} t^{n + l - 2m} \right]_+ t^{2m -2l + 1} dt \\ 
 \geq & \ \int_{t_0}^{t_2} \left[ \left( \frac{\beta l }{\lambda} \right)^2 - \frac{l b}{2 \lambda} t^{n + l - 2m} \right]_+ t^{2m -2 l + 1} dt + \delta \int_{t_2}^{+\infty} t^{2m - 2l + 1} dt \ = +\infty,
\end{align*}
thus \ref{V} is not fulfilled.

3. It is only left to consider the case 
$$n + l = 2m.$$

Condition \ref{V} becomes
\[
\int_{t_0}^{+\infty} \left[ \left( \frac{\beta l }{\lambda} \right)^2 - \frac{l b}{2 \lambda} \right]_+ t^{2m - 2l + 1} dt < +\infty.
\]
If $b \geq \frac{2 \beta^2 l }{\lambda}$, then it is fulfilled. Otherwise, since $2m - 2l + 1 = n-l+1 \geq -1$, it is not fulfilled.

Summarising, all convergence statements in Theorem \ref{CFVG} and Theorem \ref{CT} hold in the two settings

\begin{enumerate}
\item $\alpha > 1$, $\beta = 0$, $\alpha - 3 > n$, $0 \leq l \leq 1$, and $b, \lambda > 0$;

\item $\alpha > 1$, $\beta > 0$, $\alpha - 3 > n$, $0 \leq l \leq 1$, $0 \leq m \leq n + 1$, $b > \frac{(m + 1) (\alpha - m - 2) \beta}{(\alpha - 3 - n)t_0^{n - m + 1}}$, $\lambda > 0$, and either $2m < n + l$, or $2m = n + l$ and $b \geq \frac{2l \beta^2}{\lambda}$.
\end{enumerate}

\begin{remark}\label{remark2}\rm
Theorem \ref{CFVG} is providing for the choices $b(t) = b t^n$ and $\lambda(t) = \lambda t^l$ the following convergence rates
\[
\Phi(\prox\nolimits_{\lambda(t)\Phi}(x(t))) - \Phi^* = o \left( \frac{1}{t^{n+2}} \right), \quad
\|\prox\nolimits_{\lambda(t)\Phi}(x(t)) - x(t) \| \ = \ o \left( \frac{1}{t^{\frac{n}{2} + 1 - \frac{l}{2}}} \right)
\]
and 
\[
\| \nabla \Phi_{\lambda(t)} (x(t)) \| \ = \ o \left( \frac{1}{t^{\frac{n}{2} + 1 + \frac{l}{2}}} \right),
\]
as $t \to +\infty$. Clearly, the bigger the $n$ is the faster the convergence is. On the other hand, concerning the exponent $l$ things are a bit more complicated: we may gain in one case, but inevitably lose in the other. Interesting case is when $l = 0$, which corresponds to $\lambda$ being a constant function. In this case, one can notice a balance between accelerating the convergence of $\| \nabla \Phi_{\lambda(t)} (x(t)) \|$ and slowing the latter for $\|\prox\nolimits_{\lambda(t)\Phi}(x(t)) - x(t) \|$, since none of them are affected by $l$ anymore.

\end{remark}

\section{Numerical examples}

In this section we will conduct series of experiments to investigate the influence of the system parameters $\lambda$, $\beta$ and $b$ on the convergence behaviour of dynamical system. We will successively fix two of them and vary the last one in order to do so. For the numerical experiments we will restrict ourselves to the polynomial choices addressed in the previous section $\lambda(t) = t^l$, $\beta(t) = t^m$, $b(t) = bt^n$ with $b = \frac{(m + 1) (\alpha - m - 2) \beta}{(\alpha - 3 - n) t_0^{n - m + 1}} + 1$, as well as $x(t_0) = x_0 = 10$, $\dot x(t_0) = 0$,  and $t_0 = 1$.

\subsection{The influence of $b$ on the dynamical behaviour}

First let us choose as objective function $\Phi : \mathbb{R} \rightarrow \mathbb{R}_+, \Phi(x) = |x|$, fix $m = 0$, $\alpha = 9$ and $l = 1$, and vary $n$. 

\begin{figure}[H]
     \centering
     \begin{subfigure}[b]{0.32\textwidth}
         \centering
         \includegraphics[width=\textwidth]{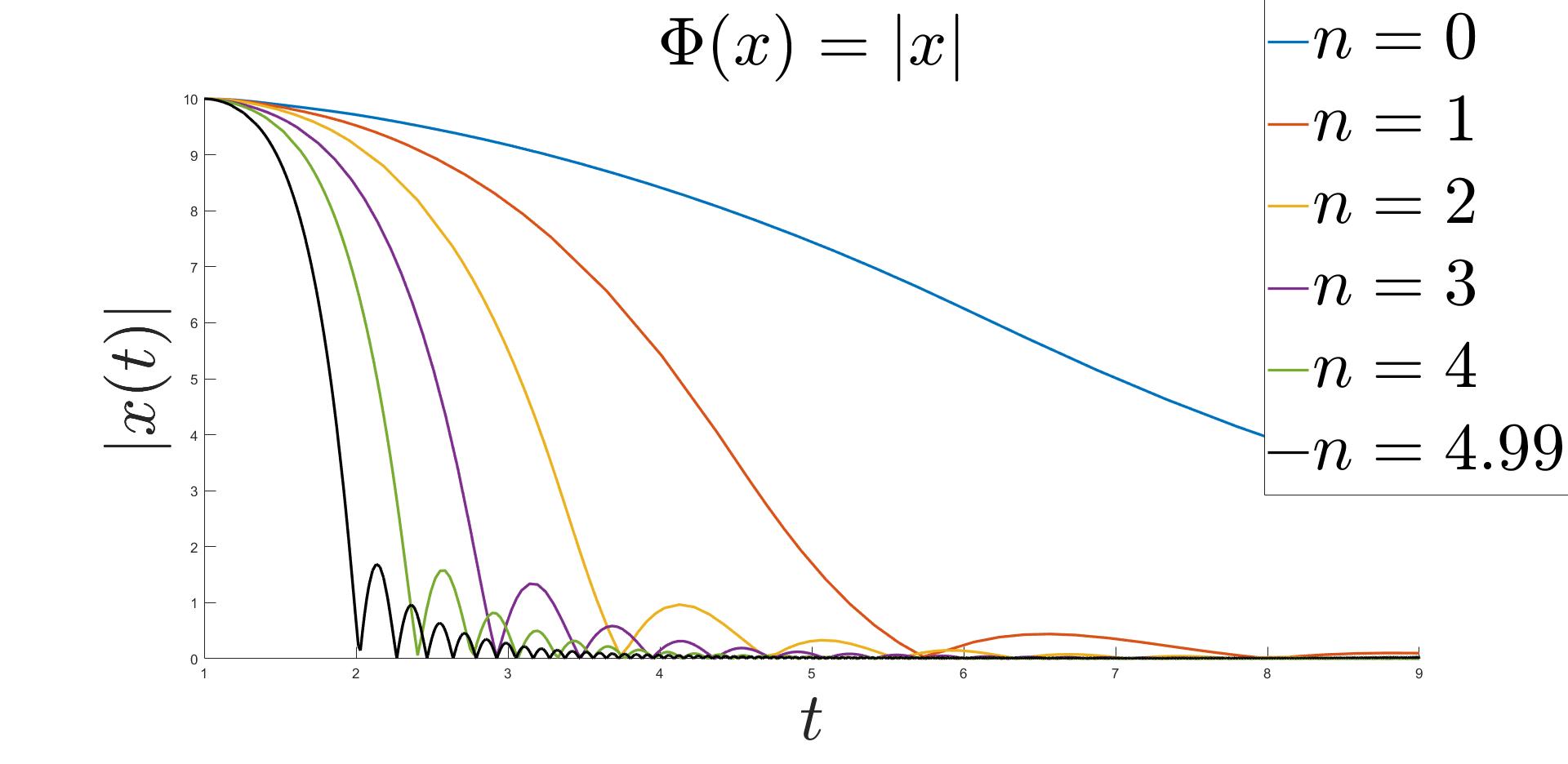}
         \caption{Trajectories}
     \end{subfigure}
     \hfill
     \begin{subfigure}[b]{0.32\textwidth}
         \centering
         \includegraphics[width=\textwidth]{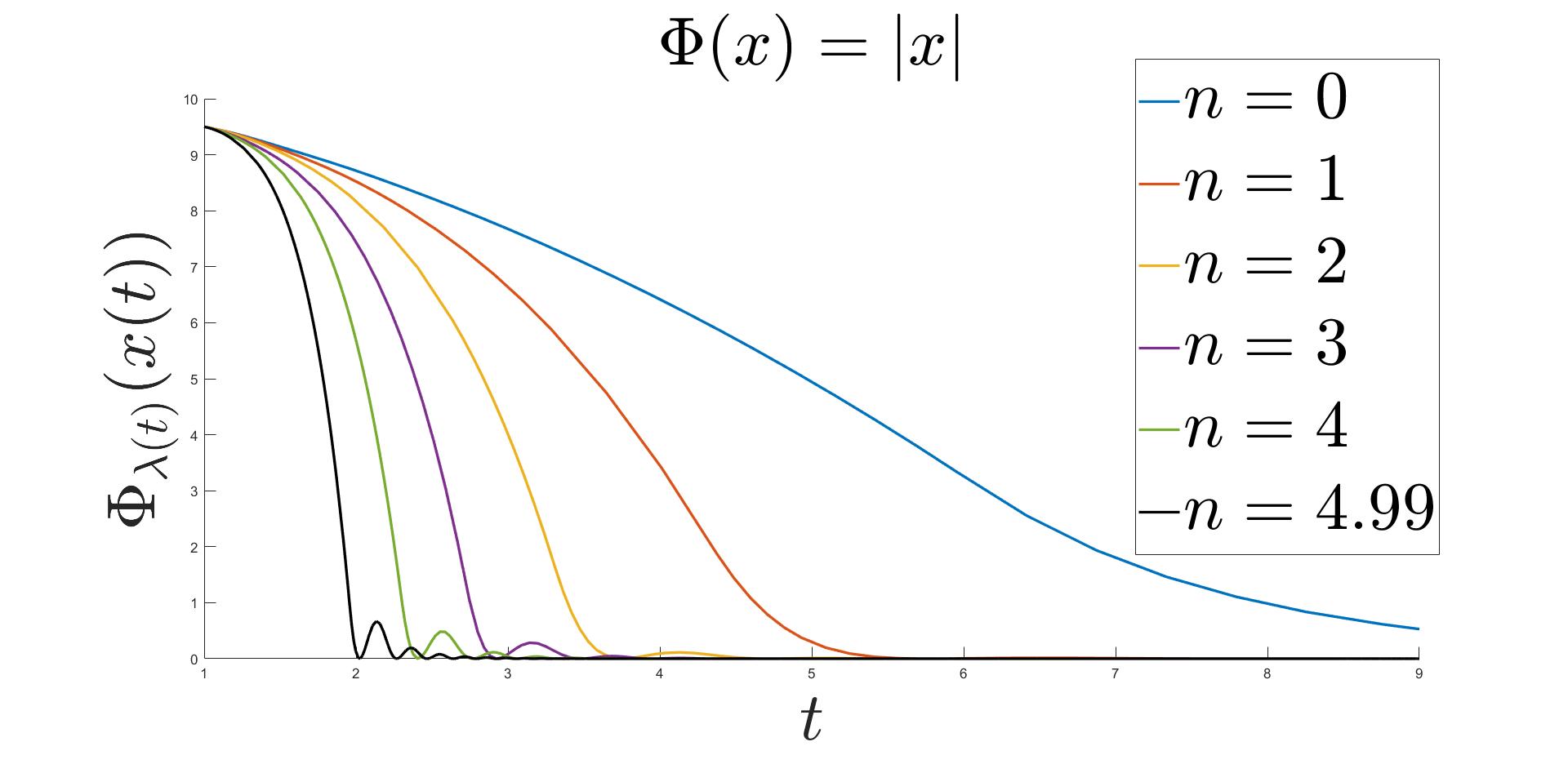}
         \caption{Moreau envelope values}
     \end{subfigure}
     \hfill
     \begin{subfigure}[b]{0.32\textwidth}
         \centering
         \includegraphics[width=\textwidth]{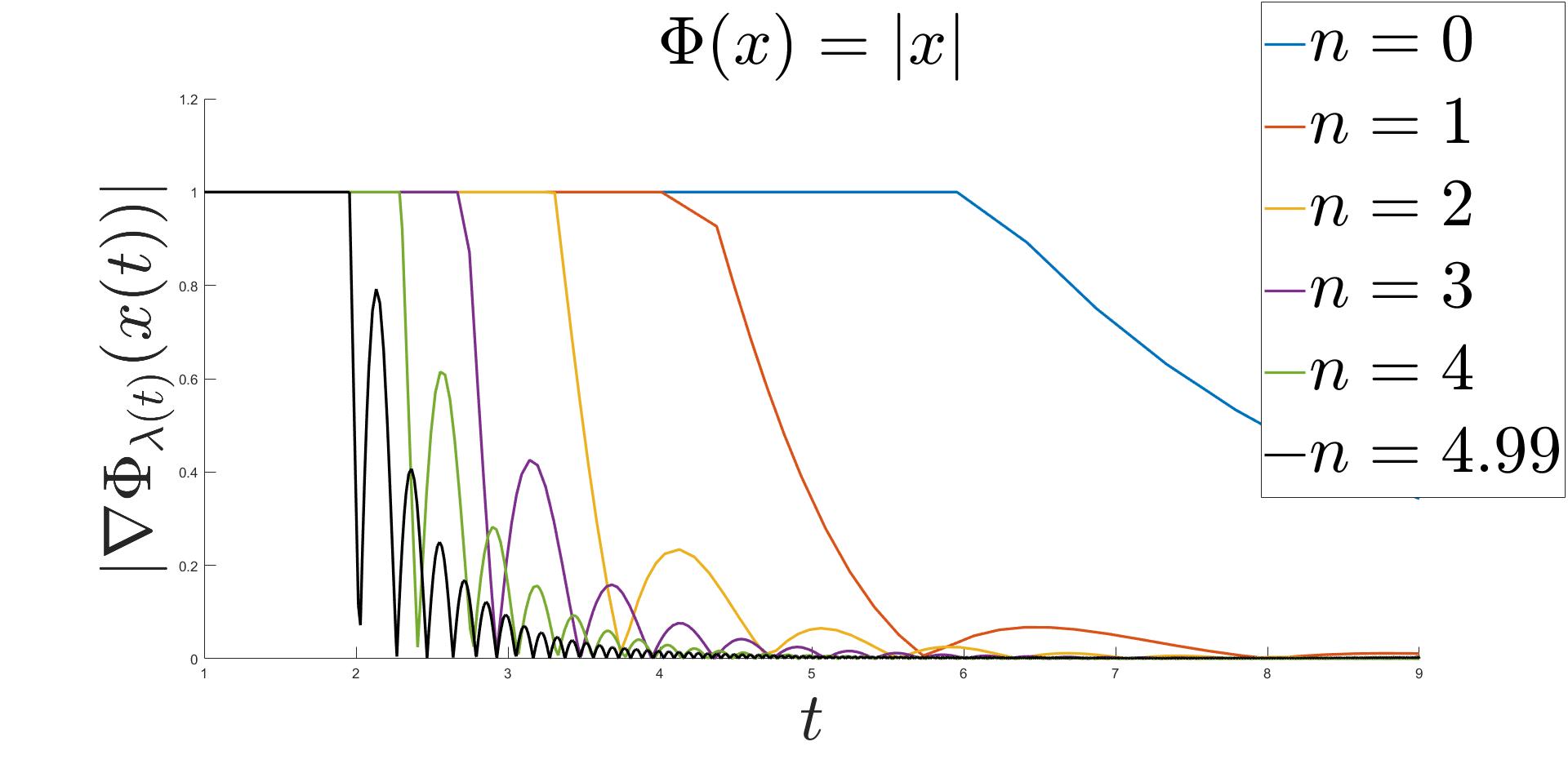}
         \caption{Moreau envelope gradient}
     \end{subfigure}
        \caption{$m = 0$, $\alpha = 9$ and $l = 1$}
\end{figure}

In Figure 1 we clearly see that the faster the exponent of the function $b$ grows the faster the convergence of the function values of the Moreau envelope and its gradient are, starting with the slowest pace for $n = 0$ and accelerating until $n = 4.99$, confirming the theoretical convergence rates. In addition, the increase in the exponent of $b$ seems to improve the convergence behaviour of the trajectory, too.  Fast growing exponents for $b$ will improve the convergence greatly, however, as seen in the previous section, they are limited by the upper bound value $\alpha-3$.

\subsection{The influence of $\lambda$ on the dynamical behaviour}

For the same objective function as in the previous subsection, we study the behaviour of the dynamics when varying the exponent $l$ to investigate the influence of the function $\lambda$. To this end we fix $m = 0$, $\alpha = 9$ and $n = 5 < \alpha -3$, and take for $l$ three different values from $0$ to $1$. 

\begin{figure}[H]
     \centering
     \begin{subfigure}[b]{0.32\textwidth}
         \centering
         \includegraphics[width=\textwidth]{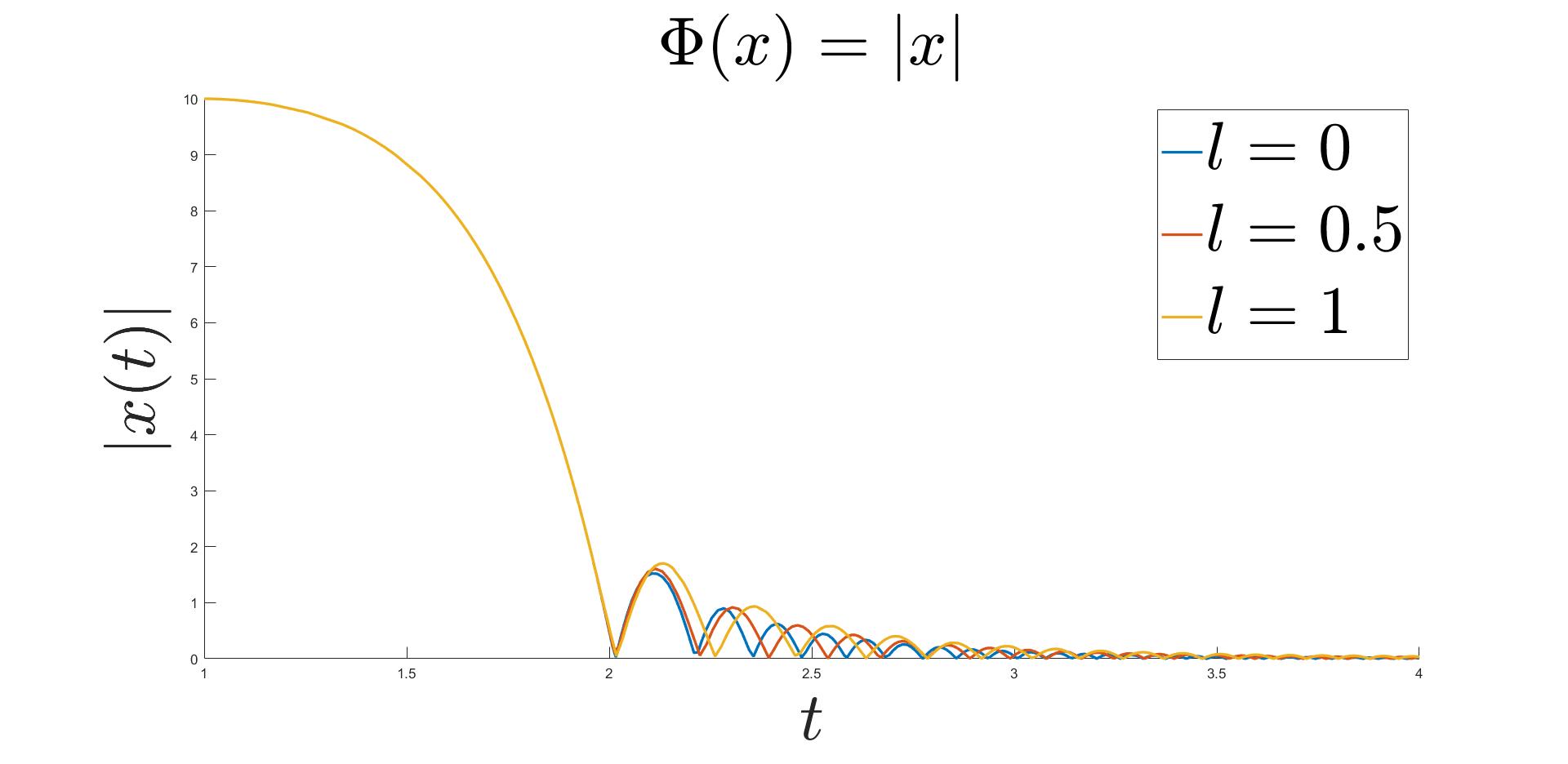}
         \caption{Trajectories}
     \end{subfigure}
     \hfill
     \begin{subfigure}[b]{0.32\textwidth}
         \centering
         \includegraphics[width=\textwidth]{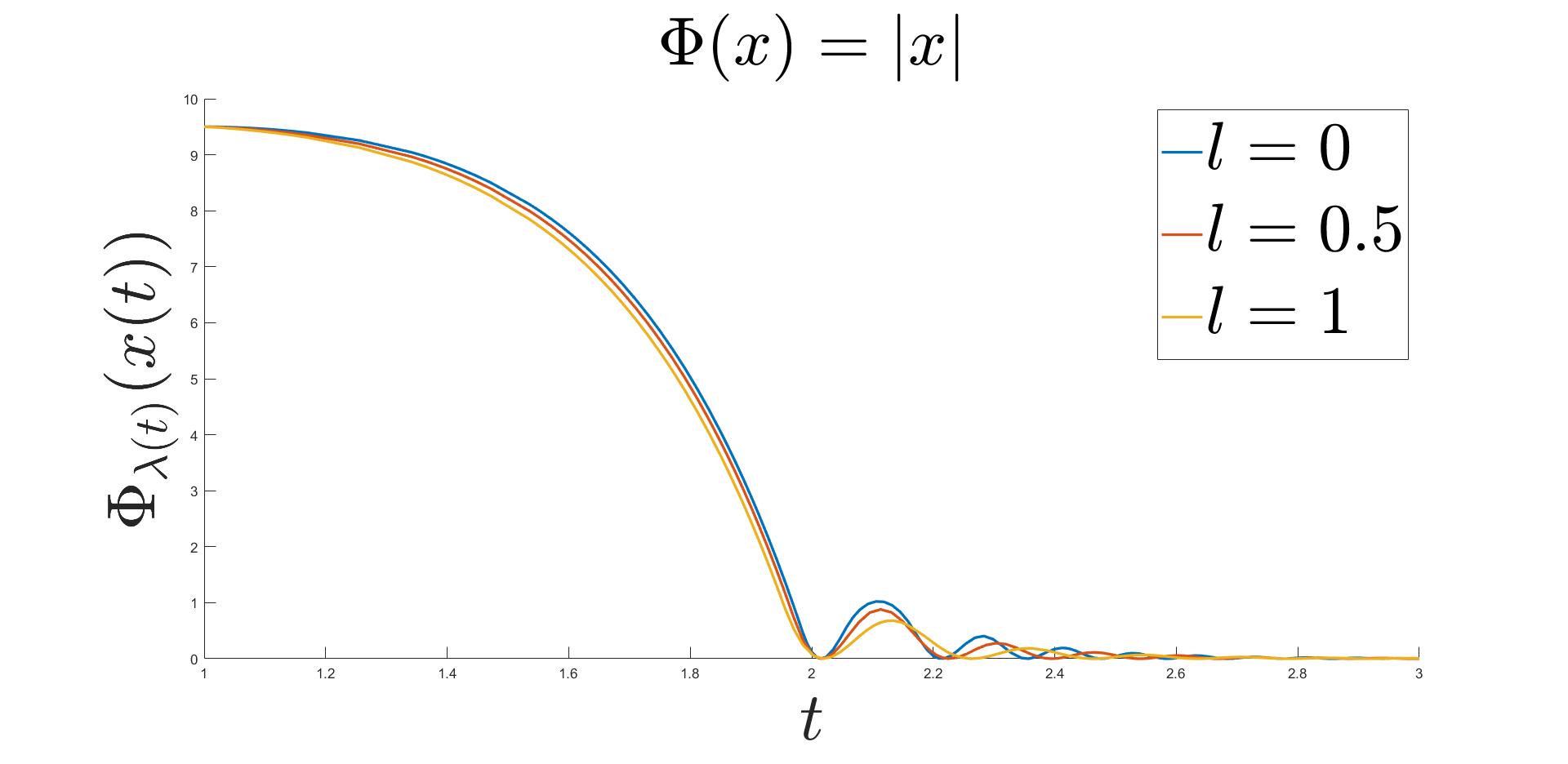}
         \caption{Moreau envelope values}
     \end{subfigure}
     \hfill
     \begin{subfigure}[b]{0.32\textwidth}
         \centering
         \includegraphics[width=\textwidth]{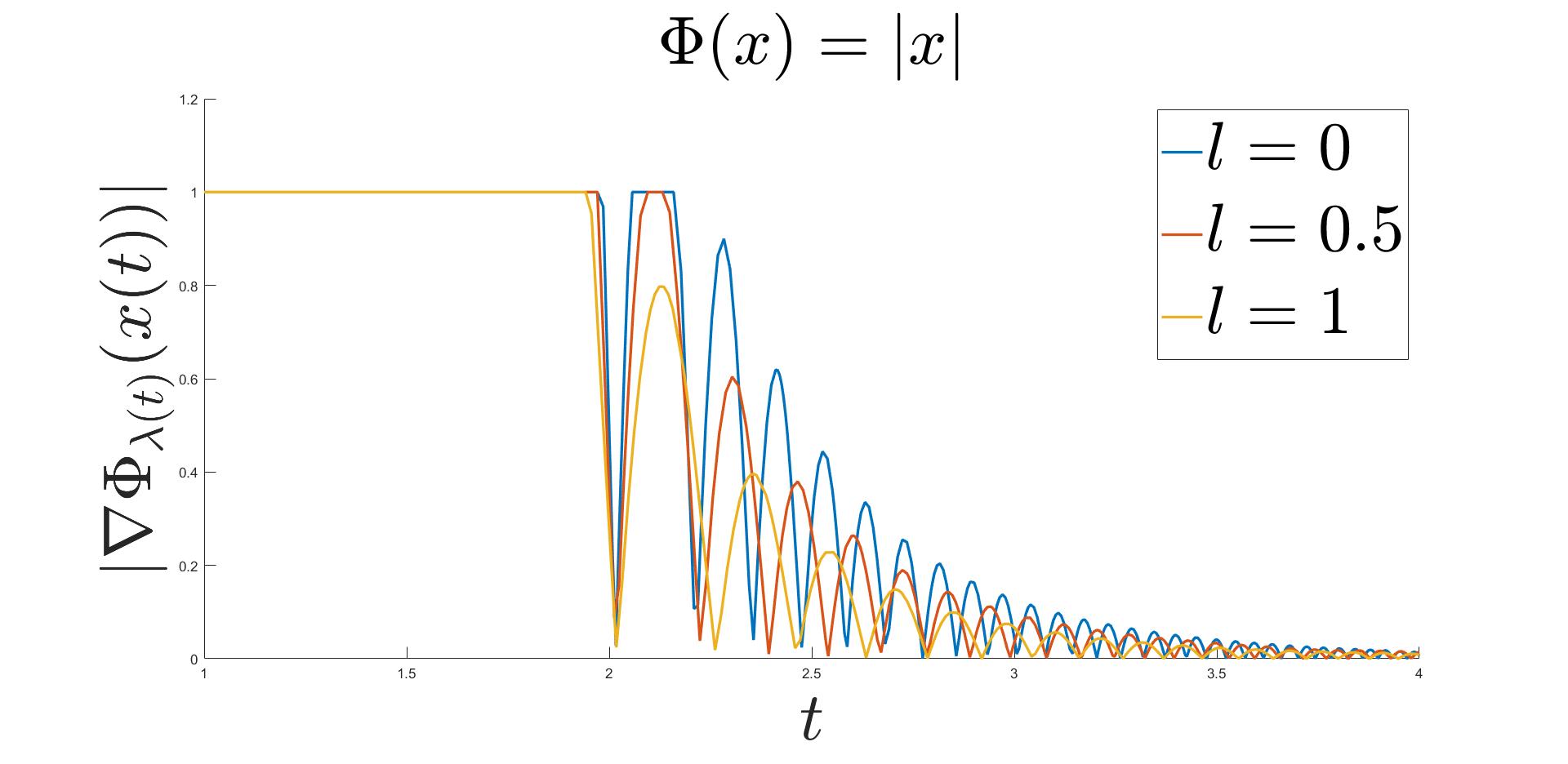}
         \caption{Moreau envelope gradient}
     \end{subfigure}
        \caption{$m = 0$, $\alpha = 9$ and $n = 5$}
\end{figure}

One can notice in Figure 2 that the convergence behaviour of the functions values of the Moreau envelope and its gradient is better the higher $l$ is, whereas, interestingly enough, for the convergence of the trajectories an opposite phenomenon takes place.  

\subsection{The influence of $\beta$ on the dynamical behaviour}

Let $\Phi : \mathbb{R} \rightarrow \mathbb{R}_+, \Phi(x) = |x| + \frac{x^2}{2}$, $\alpha = 13$, $n = 9 < \alpha -3$ and $l=1$. We vary the exponent $m$ such that $2m < n+l$ to study  the influence of the function $\beta$ on the convergence behaviour of the system.
\begin{figure}[H]
     \centering
     \begin{subfigure}[b]{0.32\textwidth}
         \centering
         \includegraphics[width=\textwidth]{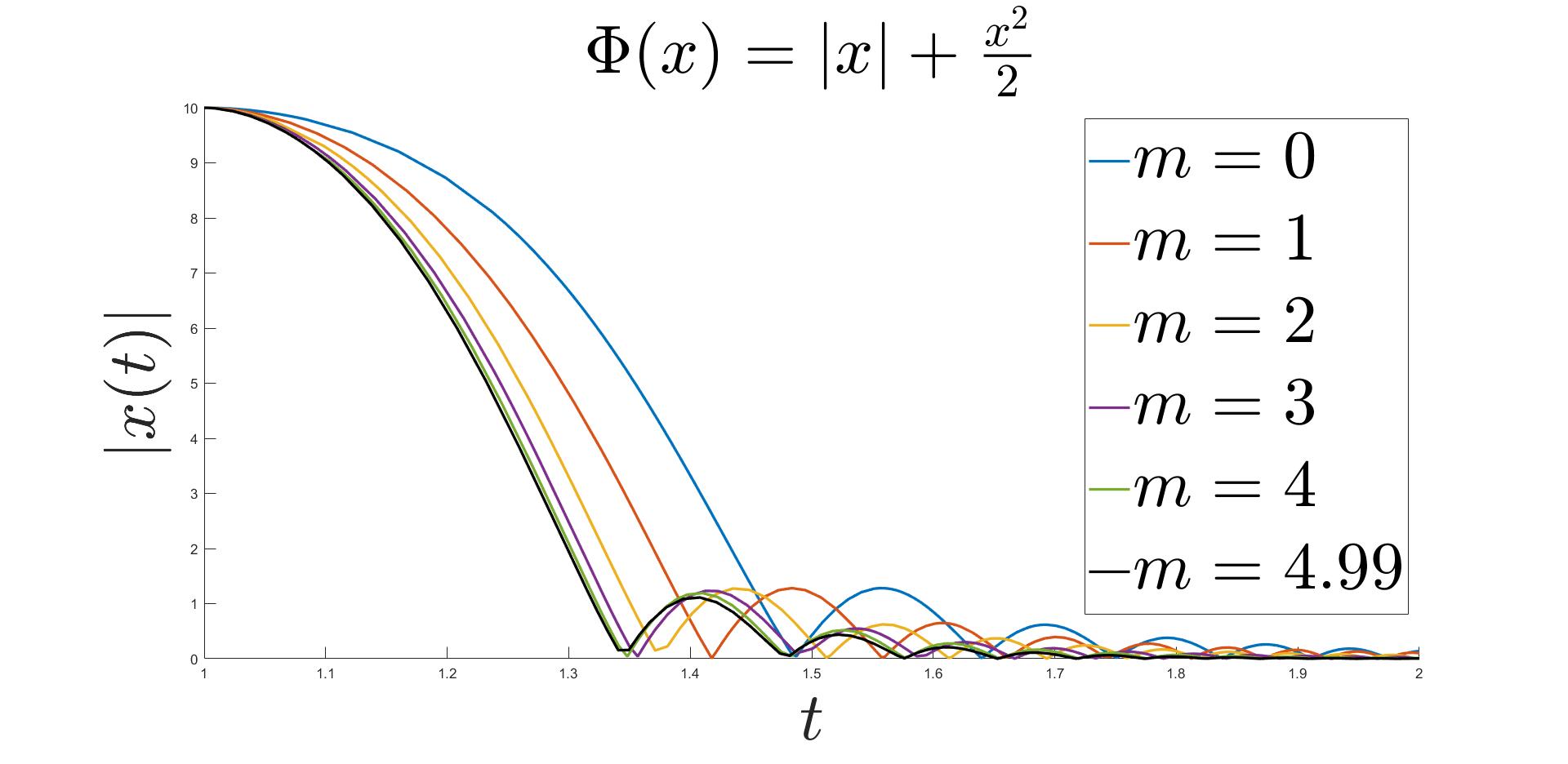}
         \caption{Trajectories}
     \end{subfigure}
     \hfill
     \begin{subfigure}[b]{0.32\textwidth}
         \centering
         \includegraphics[width=\textwidth]{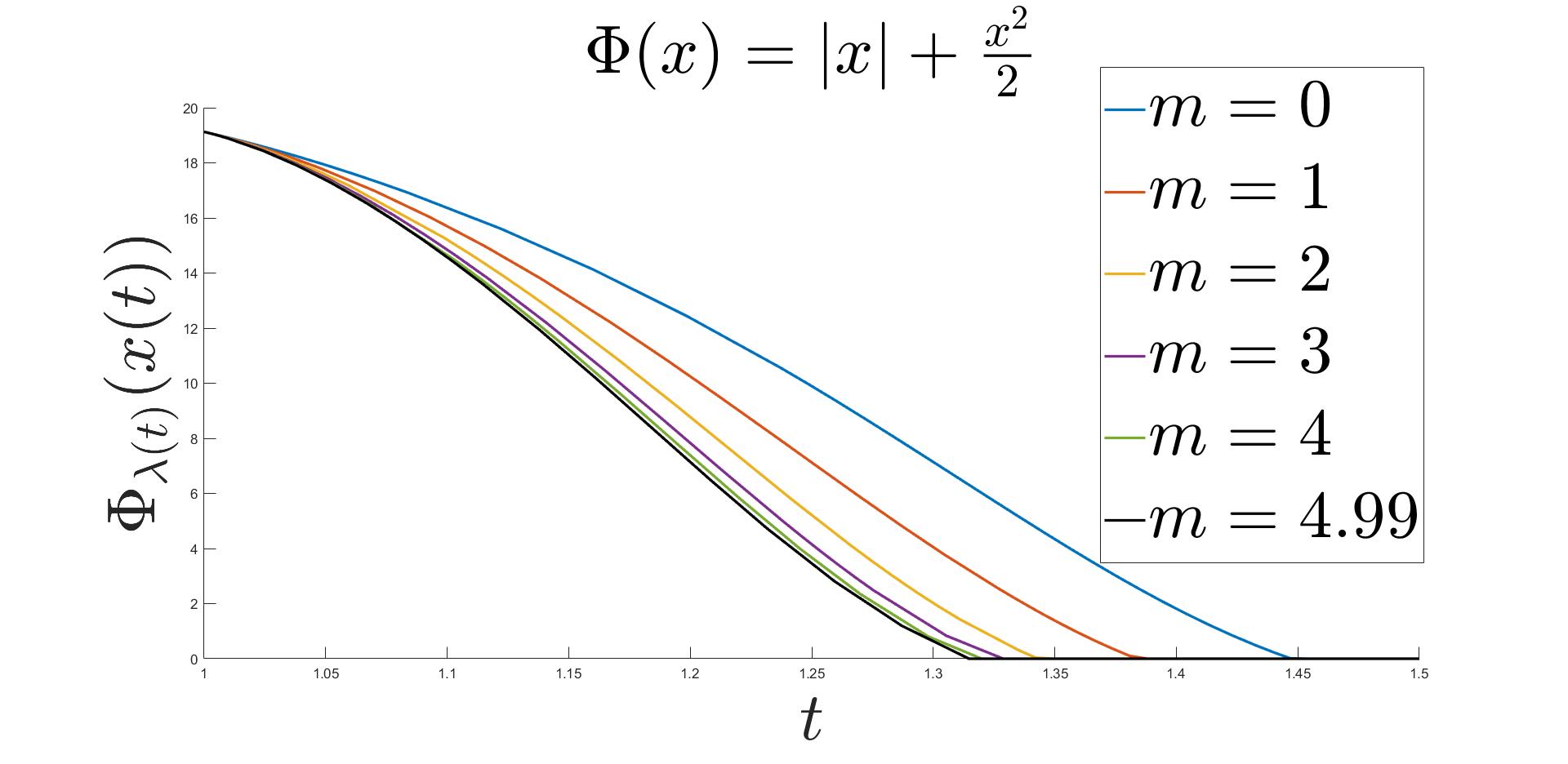}
         \caption{Moreau envelope values}
     \end{subfigure}
     \hfill
     \begin{subfigure}[b]{0.32\textwidth}
         \centering
         \includegraphics[width=\textwidth]{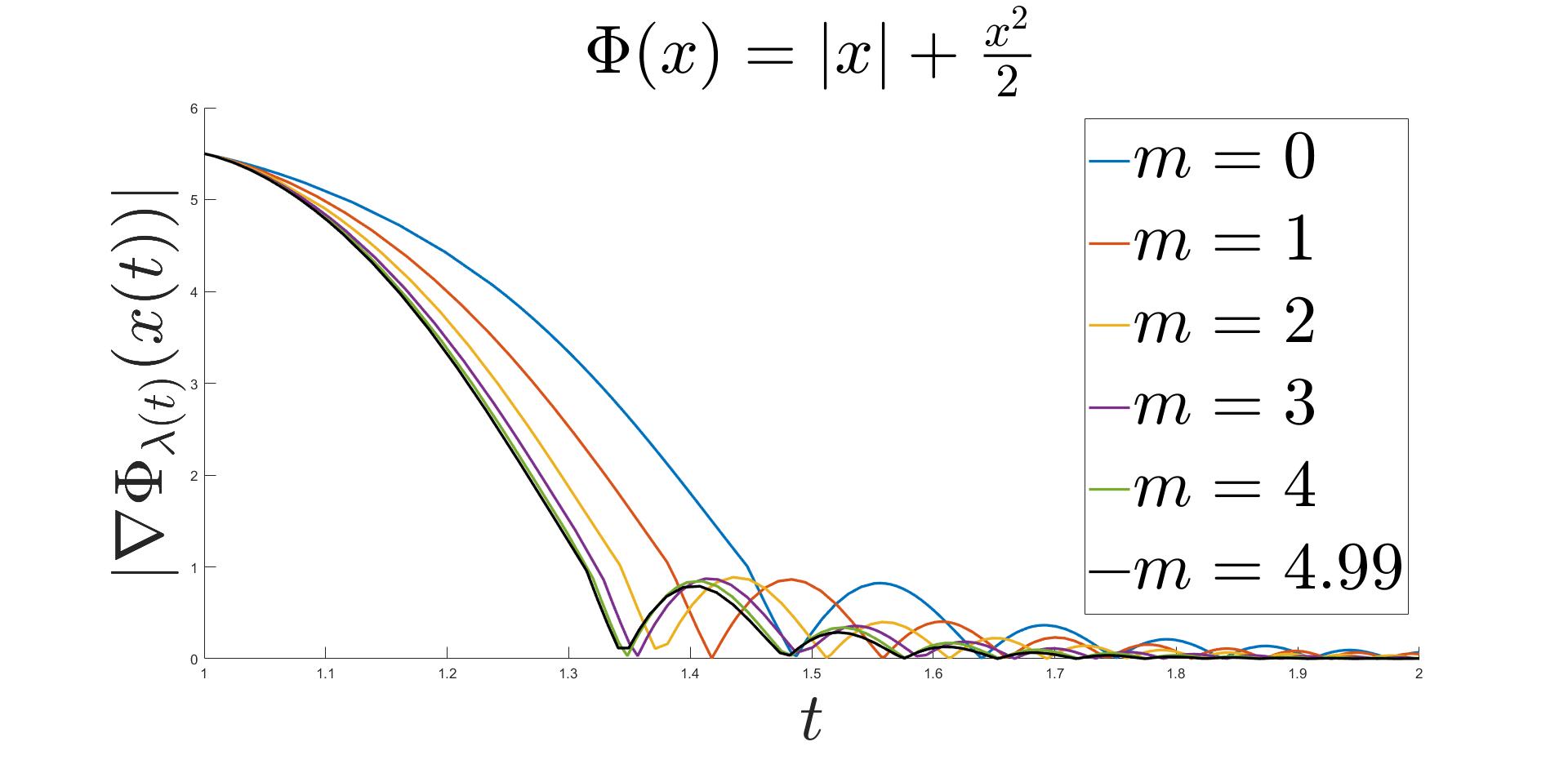}
         \caption{Moreau envelope gradient}
     \end{subfigure}
        \caption{$n = 9$, $\alpha = 13$ and $l = 1$}
\end{figure}
In Figure 3 we see that, even though $m$ does not explicitly appear in the theoretical convergence rates for the gradient of the Moreau envelope and the trajectory of the system, it influences the convergence behaviour of both of them as well as of the function values of the Moreau envelope, in the sense that these are faster the higher the values of $m$ are.

\begin{figure}[H]
     \centering
     \begin{subfigure}[b]{0.49\textwidth}
         \centering
         \includegraphics[width=\textwidth]{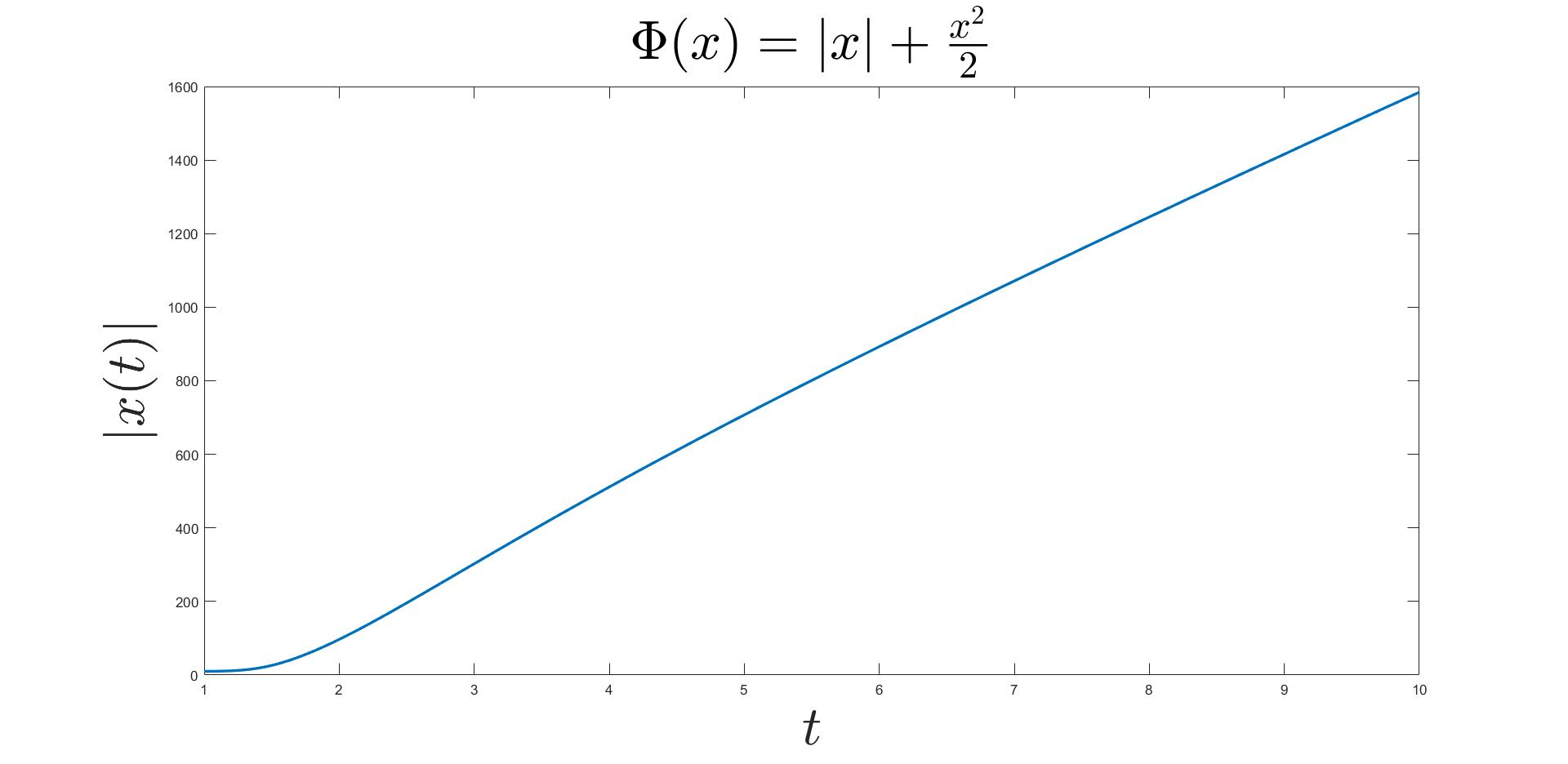}
         \caption{$n = 9$, $\alpha = 13$, $l = 1$ and $m = 12$}
     \end{subfigure}
     \hfill
     \begin{subfigure}[b]{0.49\textwidth}
         \centering
         \includegraphics[width=\textwidth]{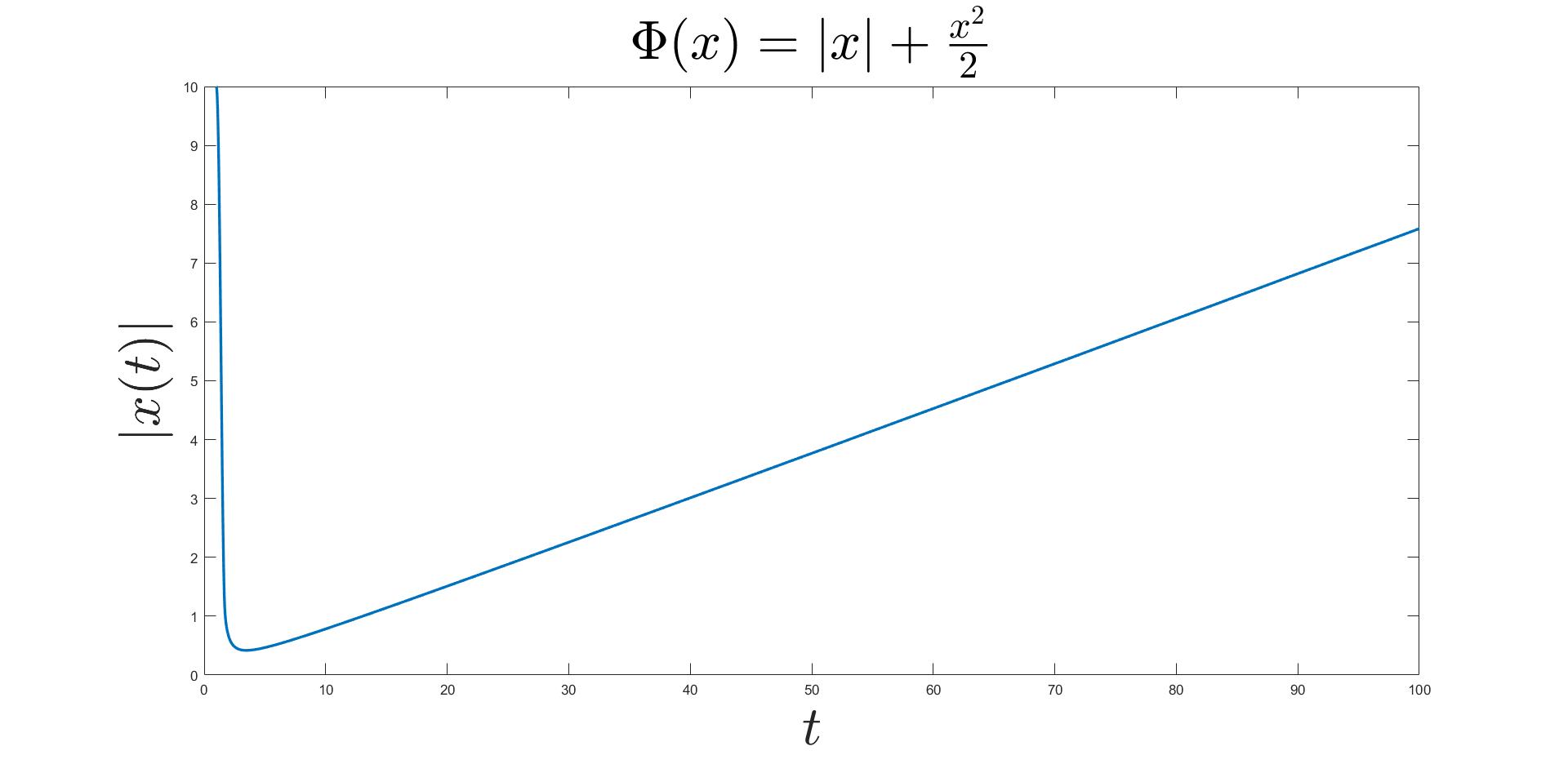}
         \caption{$n = 4$, $\alpha = 2$, $l = 4$ and $m = 6$}
     \end{subfigure}
        \caption{Divergence of the trajectories}
\end{figure}

Finally, we consider two parameter choices which lie outside the convergence setting derived in the previous section and notice that these fundamentally affects the convergence of the trajectory. In Figure 4 (a) we choose $m$ such that that condition $2m < n+l$ is violated, and in Figure (b) we choose $\alpha$ and $n$ such that the condition $\alpha - 3 > n$ is also violated. One can see that in both settings the trajectories diverge.

\section*{Appendix}

In this appendix we collect some lemmas which play an important role in the proof of the main results of the paper. For the proof of the following lemma we refer to \cite{AAS}.

\begin{lemma}\label{l}
Suppose that $f: [t_0, +\infty) \to \mathbb{R}$ is locally absolutely continuous and bounded from below and there exists $g \in L^1([t_0, +\infty), \mathbb{R})$ such that for almost all $t \geq t_0$
\[
\frac{d}{dt} f(t) \ \leq \ g(t).
\]
Then there exists $\lim_{t \to +\infty} f(t) \in \mathbb{R}$.
\end{lemma}

For the proof of the following lemma we refer to \cite{APR}.

\begin{lemma}\label{A}
Let $H$ be a real Hilbert space and $x : [t_0, +\infty) \longrightarrow \mathbb{H}$ a continuously differentiable function satisfying $ x(t) + \frac{t}{\alpha} \dot x(t) \to \ L $ as $ t \to +\infty $, with $ \alpha > 0 $ and $ L \in \mathbb{H} $. Then $ x(t) \to L $ as $ t \to +\infty $.
\end{lemma}

Finally, we state a continuous version of Opial's Lemma (see \cite{O}), which is used in the proof of the convergence of the trajectory.

\begin{lemma}\label{O}

Let $S$ be a non-empty subset of a real Hilbert space $H$ and $x: [0, +\infty) \mapsto H$ a given map. Assume that
\begin{itemize}

\item for every $z \in S$, $\lim_{t \to +\infty} \| x(t) - z \|$ exists;

\item every weak sequential cluster point of the map $x$ belongs to $S$.

\end{itemize}

Then $x(t)$ converges weakly to some element of $S$ as $t \to +\infty$.

\end{lemma}

\end{document}